\documentclass[10pt]{amsart}
\pdfoutput=1
\usepackage{amsmath}
\usepackage{amsfonts}
\usepackage{amssymb}
\usepackage{amsthm}
\usepackage{hyperref}
\usepackage{comment}
\usepackage{float}
\usepackage{youngtab}
\usepackage{ytableau}
\usepackage{rotating}
\usepackage{xcolor}
\usepackage{tikz}
\usepackage{tikz-cd}
\usepackage[paper=a4paper, margin=3.5cm]{geometry}

\def\NN{{\NZQ N}}

\def\ZZ{{\NZQ Z}}
\def\RR{{\NZQ R}}
\def\CC{{\NZQ C}}

\def\PP{{\NZQ P}}

\def\frk{\frak}               

\def\Phi{{\frk n}}
\def\Phi{{\frk N}}
\def\fU{{\mathcal U}}
\def\fQ{{\mathcal Q}}

\def\opn#1#2{\def#1{\operatorname{#2}}} 
\opn\chara{char} \opn\length{\ell} \opn\pd{pd} \opn\rk{rk}
\opn\projdim{proj\,dim} \opn\injdim{inj\,dim} \opn\rank{rank}
\opn\spn{span}\opn\Seg{Seg}
\opn\depth{depth} \opn\grade{grade} \opn\height{height}
\opn\embdim{emb\,dim} \opn\codim{codim}

\opn\Tr{Tr} \opn\bigrank{big\,rank}
\opn\superheight{superheight}\opn\lcm{lcm}
\opn\trdeg{tr\,deg}
\opn\reg{reg} \opn\lreg{lreg} \opn\ini{in} \opn\lpd{lpd}
\opn\size{size}\opn\bigsize{bigsize}
\opn\cosize{cosize}\opn\bigcosize{bigcosize}
\opn\sdepth{sdepth}\opn\sreg{sreg}
\opn\link{link}\opn\fdepth{fdepth} \opn\trdeg{trdeg} \opn\mod{mod}
\opn\spann{span}
\opn\div{div} \opn\Div{Div} \opn\cl{cl} \opn\Cl{Cl}
\opn\Spec{Spec} \opn\Supp{Supp} \opn\supp{supp} \opn\Sing{Sing}
\opn\Ass{Ass} \opn\Min{Min}\opn\Mon{Mon} \opn\dstab{dstab} \opn\astab{astab}
\opn\Syz{Syz}
\opn\Ann{Ann} \opn\Rad{Rad} \opn\Soc{Soc} \opn\Aut{Aut}
\opn\Im{Im} \opn\Ker{Ker} \opn\Coker{Coker} \opn\Am{Am}
\opn\Hom{Hom} \opn\Tor{Tor} \opn\Ext{Ext} \opn\End{End}
\opn\Aut{Aut} \opn\id{id}

\opn\nat{nat}
\opn\pff{pf}
\opn\Pf{Pf} \opn\GL{GL} \opn\SL{SL} \opn\mod{mod} \opn\ord{ord}
\opn\Gin{Gin} \opn\Hilb{Hilb}\opn\sort{sort}
\opn\S{S} \opn\dim{dim} \opn\supp{supp}\opn\trdeg{trdeg}\opn\sort{sort}
\opn\aff{aff} \opn\con{conv} \opn\relint{relint} \opn\st{st}
\opn\lk{lk} \opn\cn{cn} \opn\core{core} \opn\vol{vol}
\opn\link{link} \opn\star{star}\opn\lex{lex}
\opn\conv{conv} \opn\Ehr{Ehr}\opn\Pic{Pic}
\opn\gr{gr}

\def\pot#1#2{#1[\kern-0.28ex[#2]\kern-0.28ex]}

\opn\dirlim{\underrightarrow{\lim}}
\opn\inivlim{\underleftarrow{\lim}}
\let\to=\rightarrow

\def\Implies{\ifmmode\Longrightarrow \else
        \unskip${}\Longrightarrow{}$\ignorespaces\fi}
\def\implies{\ifmmode\Rightarrow \else
        \unskip${}\Rightarrow{}$\ignorespaces\fi}
\def\iff{\ifmmode\Longleftrightarrow \else
        \unskip${}\Longleftrightarrow{}$\ignorespaces\fi}

\let\:=\colon
\def\CQ{CQ}
\newtheorem{Theorem}{Theorem}[section]
 \newtheorem{Lemma}[Theorem]{Lemma}
 \newtheorem{Corollary}[Theorem]{Corollary}
 \newtheorem{Proposition}[Theorem]{Proposition}
 \newtheorem{Remark}[Theorem]{Remark}
 
 \newtheorem{Example}[Theorem]{Example}
 
 \newtheorem{Definition}[Theorem]{Definition}
 
 \newtheorem{Conjecture}[Theorem]{Conjecture}

\DeclareMathOperator{\diag}{diag}

\def\CC{\mathbb{C}}

\def\RR{\mathbb{R}}
\def\ZZ{\mathbb{Z}}
\def\PP{\mathbb{P}}
\def\NN{\mathbb{N}}

\begin{document}
 \title {Geometry of the Gaussian graphical model of the cycle}
\keywords{Gaussian graphical model, complete quadrics, intersection theory, projective degree }
 
 \author{Rodica Andreea Dinu}
\address{%
	University of Konstanz, Fachbereich Mathematik und Statistik, Fach D 197 D-78457 Konstanz, Germany, and Institute of Mathematics ``Simion Stoilow" of the Romanian Academy, Calea Grivitei 21, 010702, Bucharest, Romania}
	\email{rodica.dinu@uni-konstanz.de}

 \author{Mateusz Micha{\l}ek}
 \address{%
	University of Konstanz, Fachbereich Mathematik und Statistik, Fach D 197 D-78457 Konstanz, Germany}
	\email{mateusz.michalek@uni-konstanz.de}
 
\author{Martin Vodi\v{c}ka}
\address{%
	Max Planck Institute for Mathematics in the Sciences, Inselstrasse 22, 041 03 Leipzig, Germany and University of Konstanz, Fachbereich Mathematik und Statistik, Konstanz, Germany}
	\email{martin.vodicka@uni-konstanz.de}

\thanks{MM is supported by the DFG grant nr 467575307, RD is supported by the Alexander von Humboldt Foundation.}
\maketitle

 \begin{abstract}
We prove a conjecture due to Sturmfels and Uhler concerning the degree of the projective variety associated to the Gaussian graphical model of the cycle. We involve new methods based on the intersection theory in the space of complete quadrics. 
 \end{abstract}

 \maketitle

\section{Introduction}

Algebraic statistics employs techniques in algebraic geometry, commutative algebra and combinatorics, to address problems in statistics and its applications. The philosophy of algebraic statistics is that statistical models are algebraic varieties.\\
\indent One of the central problems of interest for statisticians is the following: Given data obtained after an experiment, coming from a probability distribution in our model, estimate the parameters that best explain the observed data. This problem is usually approached by maximizing the likelihood function for the given data, i.e. by computing the maximum likelihood estimate MLE. The complexity of finding  the MLE is measured by the maximum likelihood degree (ML-degree). The study of the ML-degree is a very active topic in algebraic statistics, see for example \cite{carlos, huh, huh2, orlando, uhler, sonja, bsp, Seth}.

{\it A multivariate Gaussian distribution} in $\RR^n$ is determined by the mean vector $\mu \in \RR^n$ and a positive semidefinite $n\times n$ matrix $\Sigma$ called the \emph{covariance matrix}. Its inverse matrix $K=\Sigma^{-1}$ is also positive semidefinite and known as the \emph{concentration matrix} of the distribution. We assume that $K$ belongs to a fixed linear subspace $\Lambda$, in other words the concentration matrix can be written as a linear combination $K=\lambda_1 K_1+\dots+\lambda_d K_d$ of some fixed linearly independent matrices $K_1, \dots, K_d$. In this case, the statistical model is called a \emph{linear concentration model}. The linear concentration model associated to $\Lambda$ is the variety $\Lambda^{-1}:=\overline{\{K^{-1}:K\in \Lambda\}}$. This statistical model was introduced by Anderson \cite{anderson}.


Let $\mathbb{S}^n$ denote the space of $n\times n$ symmetric matrices and the projection $\pi: \mathbb{S}^n \rightarrow \mathbb{S}^n/{\Lambda^{\perp}}$.
We assume that $\Lambda$ intersects the cone of positive definite matrices $PD_n$. The cone of concentration matrices is defined as $K_{\Lambda}= \Lambda \cap PD_n$. The covariance matrices allowed by the model form a semialgebraic set $K_{\Lambda}^{-1}=\{M^{-1}\:\; M\in K_{\Lambda}\}$. Both $K_{\Lambda}$ and $K_{\Lambda}^{-1}$ are subsets of the space of symmetric $n\times n$ matrices.

\begin{Theorem}(\cite[Corollary 1]{bc})
For the linear concentration model, the MLE is given by the unique matrix $\Sigma_0 \in K_{\Lambda}^{-1}$ such that $\pi(\Sigma_0)=\pi(S)$.
\end{Theorem}

\begin{Definition}
The ML-degree of a model is the degree of the dominant rational map $\pi: \overline{K}_{\Lambda}^{-1} \dashrightarrow \PP(\mathbb{S}^n/{\Lambda^{\perp}})$.
\end{Definition}

The {\it Gaussian graphical model} is a special case of the linear concentration model, where $\Lambda\subset \mathbb{S}^n$ is defined by the vanishing of some entries of the matrix. For a positive integer $n$, let $C_n$ be the undirected cycle with vertex set $[n]$ and edges $(1,2), (2,3), \dots, (n-1, n), (n,1)$. The undirected Gaussian graphical model associated with the graph $C_n$ is the family of multivariate normal distributions with covariance matrix $\Sigma$ in the set
$$
\{\Sigma \in PD_n \:\; (\Sigma^{-1})_{ij} =0, (i,j) \notin E(C_n),i\neq j\}.
$$
The maximum likelihood estimation for this model is a matrix completion problem, see~\cite[Theorem 2.1.14]{drton}.

\begin{Example}
Let $n=5$. Then for the covariance matrix $\Sigma$, we have \[\Sigma^{-1}=
\begin{pmatrix}
a_1&b_1&0&0&b_5\\
b_1&a_2&b_2&0&0\\
0&b_2&a_3&b_3&0\\
0&0&b_3&a_4&b_4\\
b_5&0&0&b_4&a_5\\
\end{pmatrix}. \]
\end{Example}

It is known that, for a generic linear subspace $\Lambda \subset \mathbb{S}^n$, the ML-degree of the linear concentration model given by $\Lambda$ is given by the degree of the projective variety $\Lambda^{-1}$, see for example \cite[Theorem 1]{bc}, \cite[Corollary 2.6]{mateusz1}. However, whenever $\Lambda$ is not generic, and this is the case for the Gaussian graphical models, this equality does not necessarily hold. For a detailed relation, involving the Segre classes, we refer to \cite{Kathlen}.
In the case of $n$-cycle there are conjectured formulas for both ML-degree and degree of the variety $\Lambda^{-1}$,  \cite[Section 7]{drton},\cite[Conjecture 4.6]{bc}. In this article we focus on the latter, precisely we prove the following theorem which was conjectured by Sturmfels and Uhler.

\begin{Theorem}(Conjecture 4.6 \cite{bc})\label{variety}
The degree of the projective variety associated to the Gaussian graphical model described by the $n$-cycle is equal to
\[
\frac{n+2}{4}{2n \choose n}-3\cdot 2^{2n-3}.
\]
\end{Theorem}

We prove Theorem~\ref{variety} by using a new approach based on the intersection theory of the variety of complete quadrics, due to Micha{\l}ek, Monin and Wi\'sniewski \cite{mateusz1}; see also \cite{mateusz2}. So far the variety of complete quadrics was only used to provide information about generic linear concentration models. Our article provides methods to address concrete models, which are most sought for in algebraic statistics.

In Section~\ref{quadrics} we give an introduction to the variety of complete quadrics. Let $L(C_n)$ be the linear subspace given by the $n$-cycle. We consider the closure of the variety obtained by inverting all the regular matrices in $L(C_n)$, and we denote it by $L(C_n)^{-1}$. In Corollary~\ref{CQ_formula}, we express the degree of $L(C_n)^{-1}$ in terms of its class in the cohomology ring of the variety of complete quadrics. In Section~\ref{low rank} we provide the description of low-rank matrices in $L(C_n)$. In Section~\ref{geometry} we work with the strict transform $X(C_n)$ of the space $L(C_n)$ in the space of complete quadrics. We geometrically describe its intersections with divisors in the space of complete quadrics. We use the results from Section~\ref{geometry} in Section \ref{section:proof} to prove formulas in the Chow ring of the space of complete quadrics, which will allow us to conclude the desired result.


\section{Complete quadrics}\label{quadrics}
We start by following the approach of \cite{mateusz2}, where the authors compute the ML-degree of general linear concentration model as intersection product in the space of complete quadrics. Here we give a brief introduction to the space of complete quadrics. More details can be consulted in \cite{laksov, thaddeus, massarenti, concesi}.

First, we pass to the complex projective space of symmetric matrices. Let $V$ be an $n$-dimensional complex vector space. We will always assume that $n\ge 3$. Then the space of $n\times n$ symmetric matrices is $\PP(S^2 (V))$.  The space of complete quadrics is a compactification of the space of regular symmetric matrices $\PP(S^2 (V))^\circ$. We begin by giving two equivalent definitions of the space of complete quadrics.
For $A\in S^2 (V)$ we denote by $\bigwedge^r A$ the corresponding form in $S^2 (\bigwedge^r V)$. In coordinates, $\bigwedge^r A$ is the matrix given by all $r\times r$ minors of $A$. In particular, $\bigwedge^{n-1} A$ is the inverse of $A$ (up to a scalar multiple).

\begin{Definition}\label{CQ1}
The space of complete quadrics $\CQ(V)$ is the closure of\\ $\phi(\PP(S^2(V))^\circ)$, where
$$
\phi : \PP(S^2(V))^\circ  \to \PP(S^2(V)) \times \PP\left( S^2\left(V\wedge V\right) \right) \times\ldots\times \PP \left(S^2\left(\bigwedge^{n-1} V\right) \right)
$$
is given by
$$
[A]\mapsto \left([A],[\bigwedge^2 A],\ldots, [\bigwedge^{n-1} A]\right).
$$
\end{Definition}

 The space of complete quadrics comes with natural projections $$\pi_i:\CQ(V)\rightarrow\PP(S^2(\bigwedge^{i} V)).$$

An equivalent definition is as follows.
\begin{Definition}\label{CQ2}
The space of complete quadrics $\CQ(V)$ is the successive blow-up of $\PP(S^2(V))$:
$$
\CQ(V) = Bl_{\widetilde D^{n-1}} Bl_{\widetilde D^{n-2}} \ldots Bl_{D^1} \PP(S^2(V)),
$$
where $\widetilde D^{i} $ is the strict transform of the space of symmetric matrices of rank at most $i$
under the previous blow-ups.
\end{Definition}

We will be interested in divisors on $\CQ(V)$. There are two natural types of divisors. Divisors $L_1,\dots,L_{n-1}$ are pullbacks of hyperplanes under $\pi_i:\CQ(V)\rightarrow\PP(S^2(\bigwedge^{i} V))$. Divisors $S_1,\dots,S_{n-1}$ are classes of the strict transforms $E_1,\dots,E_{n-1}$ of the exceptional divisors of blow-ups from our second definition, i.e.~a general point of $E_i$ maps by $\pi_1$ to a symmetric matrix of rank $i$. Alternatively $E_i$ is the set of all points $(A_1,\dots,A_{n-1})\in\CQ(V)$ for which $\rank A_i=1$.

\begin{Proposition}(\cite{schubert}, \cite[Proposition 3.6, Theorem 3.13]{massarenti})\\
We have the following relations between $S_i$ and $L_i$ in $Pic(\CQ(V))$:

$$
S_i= -L_{i-1}+2L_i-L_{i+1},
$$
where we formally put $L_0=L_n=0$.
\end{Proposition}

As a consequence we can also express $L_i$ in terms of $S_i$. We will need the following relations:

\begin{equation}
L_{n-1}=\frac{1}{n}\Big(S_1 +2S_2 + \ldots + (n-1)S_{n-1}\Big).
\end{equation}
\begin{equation}\label{relations}
L_{n-1}=\frac{1}{n-i}\Big(L_i +S_{i+1} + \ldots + (n-i-1)S_{n-1}\Big).
\end{equation}

Note that the first equation is a special case of the second one for $i=0$, where we again formally put $L_0=0$.

We denote by $L(\Gamma)\subset \PP(S^2 V)$ the linear subspace for the Gaussian graphical model of the graph $\Gamma$, i.e.~the locus of symmetric matrices $A$ for which $A_{ij}=0$ when $i\neq j$ and $\Gamma$ does not have an edge between the $i$-th and $j$-th vertex. The case we will be mostly interested in is $\Gamma=C_n$, however many of our methods work for arbitrary $\Gamma$. Our main goal is to compute the degree of $L(C_n)^{-1}$ which is the closure of the variety obtained by inverting all regular matrices in $L(C_n)$. Let $X(\Gamma)$ be the strict transform of $L(\Gamma)$ in the space of complete quadrics, i.e. the closure of the image of $\phi:L^\circ(\Gamma)\rightarrow\CQ(V)$, where $L^\circ(\Gamma)$ is the space of all regular matrices in $L(\Gamma)$ and $\phi$ is as in the Definition \ref{CQ1}. We denote by $[X(\Gamma)]$ its class in the cohomology ring of $\CQ(V)$.

\begin{Proposition}\label{Prop: deg} The following formula holds
$$\deg(L(\Gamma)^{-1})=[X(\Gamma)]L_{n-1}^{\dim(L(\Gamma)^{-1})}.$$
\end{Proposition}
\begin{proof}
The degree of $L(\Gamma)^{-1}$ can be computed as the number of intersection points of $L(\Gamma)^{-1}$ with $\dim(L(\Gamma)^{-1})$ general hyperplanes. Since the hyperplanes are general, the intersection points will be just inverses of regular matrices from $L(\Gamma)$. On the other hand $[X(\Gamma)]L_n^{\dim(L(\Gamma)^{-1})}$ is also the number of points $(A_1,\dots,A_{n-1})\in X(C_n)$ which satisfy $\dim(L(\Gamma)^{-1})$ general linear conditions on $A_{n-1}$. Again, these points must be from $\phi(L^\circ(\Gamma))$, i.e.~$A_1$ is a regular matrix from $L(\Gamma)$ and $A_{n-1}$ is its inverse. Thus, both sides compute the same number.
\end{proof}
\begin{Corollary}\label{CQ_formula} The following formula holds
$$\deg(L(C_n)^{-1})=[X(C_n)]L_{n-1}^{2n-1}.$$
\end{Corollary}
\begin{proof}
Follows from Proposition~\ref{Prop: deg}, as $\dim(L(C_n)^{-1})=\dim(L(C_n))=2n-1$.
\end{proof}

Our goal is to understand the class $[X(C_n)]$ and the intersection products involving it.

\subsection{Description of points in $E_r$}\label{points-E_r}
The next step is to have a good description of the points in $E_r$. Here we give two ways how we may think about them.
For a point $P\in \CQ(V)$ we denote $P_r=\pi_r(P)\in \PP(S^2 (\bigwedge^r V))$.
From the quotient construction of the space of complete quadrics \cite{mateusz1,thaddeus}, it can be shown that
\[
E_r=\CQ_r(\fU)\times_{G(r,n)} \CQ_{n-r}(\fQ^*).
\]
In other words, points $P\in S_r$ can be viewed as a triple $(A_P,U_P,B_P)$ where  $A_P\in \CQ(U)$, $U_P\in G(r,V)$, $B_P\in \CQ(V/U)$. Since $A_P,B_P$ are also elements of some space of complete quadrics, we define $(A_P)_i$ and $(B_P)_i$ analogously as the images under the corresponding projections.

The relations between two descriptions of a point $P\in S_r$ are as follows:
\begin{itemize}

\item The second Veronese of $U_P$ equals $P_r$, which makes sense since the matrix $P_r$ is of rank $1$. 

\item Let $1\le i\le r-1$. Then $(A_P)_i$ is a natural interpretation of $P_i$ as an element of $\PP(S^2(\bigwedge^i U))\subset \PP(S^2(\bigwedge^i V))$.

\item Let $1\le i\le n-r-1$ and $W\in  G(r+i,V^*)$. We consider $P_{r+i}$ as a symmetric, bilinear form on $ \PP(\bigwedge^{r+i}V^*)$. In particular, we may evaluate it on $W$ obtaining a linear form on  $\PP(\bigwedge^{r+i}V^*)$. If $\dim(U^\perp\cap W)>i$, then the linear form $$P_{r+i}(W)=0.$$ 

We fix any vectors $v_1^*,v_2^*,\dots,v_r^*\in V^*$ such that $U^\perp,v_1^*,v_2^*,\dots,v_r^*$ generate $V^*$. Then the vectors of the form $v_1^*\wedge v_2^*\wedge\dots\wedge v_r^*\wedge u_1^*\wedge\dots\wedge u_i^*\in \bigwedge^{r+i}V^*$ for $u_1^*,\dots,u_i^*\in U^\perp$ generate the subspace of $\bigwedge^{r+i}V^*$ which contains the space of all points $w\in \bigwedge^{r+i}V^*$ for which the linear form $P_{r+i}(w)\neq 0$.

In this setting the relation between $P_{r+i}$ and $(B_P)_i$ is as follows:

$$P_{r+i}(v_1^*\wedge v_2^*\wedge\dots\wedge v_r^*\wedge u_1^*\wedge\dots\wedge u_i^*,v_1^*\wedge v_2^*\wedge\dots\wedge v_r^*\wedge {u'}_1^*\wedge\dots\wedge {u'}_i^*)=$$
$$=(B_P)_i( u_1^*\wedge\dots\wedge u_i^*, {u'}_1^*\wedge\dots\wedge {u'}_i^*).$$

Note that by making a different choice of vectors $v_1^*,v_2^*,\dots,v_r^*\in V^*$ we only multiply the billinear form $(B_P)_i$ by a constant which does not matter since we are working projectively.
\end{itemize}

 In the next sections, we will exploit the symmetries of the space $L(C_n)$. There is a natural action of dihedral group $\Aut(C_n)=D_n\subset \mathbb S_n$ on vector space $V$ such that $\sigma(e_i)=e_{\sigma(i)}$. This action naturally extends on $S^2 (V)$ and $\CQ(V)$. We notice that the space $L(C_n)$ is invariant under this action and, therefore, also the variety $X(C_n)$ is invariant. When we will work with points in $X(C_n)$, it will be enough to prove the statements after acting with a suitable $\sigma\in D_n$.

\section{Low rank matrices}\label{low rank}
In this section we describe low-rank matrices in $L(C_n)$.  We denote the entries of a matrix $A\in L(C_n)$ by $a_i,b_i$ as follows:

\[A=\begin{pmatrix}
a_1 & b_1 & 0 & 0 & \dots & b_n \\
b_1 & a_2 & b_2  & 0 & \dots & 0 \\
0   & b_2 & a_3 & b_3 & \dots & 0 \\
0 & 0 & b_3 & a_4 & \dots & 0 \\
\vdots & \vdots & \vdots & \vdots & \ddots & \vdots \\
b_n & 0 & 0 &0 & \dots & a_n
\end{pmatrix}.\]

 We denote by $Z(A)=\{i \:\; b_i=0\}$. We say that $A$ is in {\it a cyclic-block form} if $Z(A)\neq\emptyset$. We denote by $z_1<\dots<z_k$ all elements of $Z(A)$. The block $B_i$ is formed by taking the rows and the columns from $A$ indexed by $z_{i-1}+1,z_{i-1}+2,\dots,z_i$. We consider the indices $i$ modulo $k$ and $z_i$ modulo $n$. In particular, $B_1$ is formed by rows and columns $z_k+1,\dots,n,1,2,\dots,z_1$. Note that all non-zero entries of $A$ are in some block $B_i$. Moreover, all the blocks are tridiagonal. We denote the size of $B_i$ by $|B_i|$, which is $z_i-z_{i-1}$.

\begin{Lemma}
The rank of the block $B_i$ is either $|B_i|$ or $|B_i|-1$. The rank of a matrix $A$ in a cyclic-block form is the sum of the ranks of all blocks.
\end{Lemma}
\begin{proof}
Let
\[B_i=\begin{pmatrix}
a_1 & b_1 & 0 & 0 & \dots & 0 \\
b_1 & a_2 & b_2  & 0 & \dots & 0 \\
0   & b_2 & a_3 & b_3 & \dots & 0 \\
0 & 0 & b_3 & a_4 & \dots & 0 \\
\vdots & \vdots & \vdots & \vdots & \ddots & \vdots \\
0 & 0 & 0 &0 & \dots & a_{|B_i|}
\end{pmatrix}.\]

By the definition of blocks, all $b_i$ are non-zero. Then the minor obtained by deleting the first row and the last column is $b_1b_2\dots b_{|B_i|-1}\neq 0$. Hence, the rank of the block $B_i$ is at least $|B_i|-1$. The fact about the rank of $A$ is obvious.
\end{proof}

\begin{Lemma} \label{lemma:lowrank} Let $r\leq n-3$. Then all rank $r$ matrices in $L(C_n)$ are in a cyclic-block form.
\end{Lemma}
\begin{proof}
Consider a minor obtained by deleting first two rows and last two columns from $A$. This minor is $b_2b_3\dots b_{n-3}b_{n-2}b_n$. Since $r\leq n-3$, this minor has to be 0. Therefore at least one $b_i$ is 0 and hence, $A$ is in a cyclic-block form.
\end{proof}

Assume $z_1<\dots<z_k=n\in Z(A)$. If for the blocks corresponding to these zeros, we have that $\rank(B_i)=|B_i|-1$, then we say that the matrix $A$ is in a {\it block form of type $T=(t_1,\dots,t_k)$}, where $t_i=|B_i|$. Note that we do not require $\{z_1,\dots,z_k\}= Z(A)$, i.e. $Z(A)$ can have more elements.

\begin{Lemma}
Let $A\in L(C_n)$ be a singular matrix in a cyclic-block form. Then, for suitable $\sigma\in D_n$, $\sigma(A)$ is in a block form of type $T$ for some $T$.
\end{Lemma}
\begin{proof}
Since $Z(A)\neq\emptyset$, by acting with a suitable $\sigma\in D_n$, we obtain a matrix with $b_n=0$. Thus, we may restrict to the case $b_n=0$. Assume that the matrix $A$ has a block $B_j$ with $\rank B_j=|B_j|$ and $j>1$. Then we can simply merge the blocks $B_j$ and $B_{j-1}$ into one block. The rank of the obtained block is $\rank B_j+\rank B_{j-1}\ge |B_j|+|B_{j-1}|-1$. Therefore, we decrease the number of blocks and still all blocks $B_i$ are of rank $|B_i|$ or $|B_i|-1$. If $\rank B_1=|B_1|$, then we merge $B_1$ and $B_2$. We can continue merging blocks until all blocks $B_i$ are of the rank $|B_i|-1$, or we remain with just one block. However, we can not have just one block with rank $n$ since matrix $A$ is singular.
\end{proof}

As a consequence of the previous two lemmas, we get the following result:

\begin{Corollary}\label{lemma:blockform}
Let $r\le n-3$ and $A\in L(C_n)$ be a matrix of rank $r$. Then, for suitable $\sigma\in D_n$, $\sigma(A)$ is in a block form of type $T=(t_1,\dots,t_{n-r})$ for a suitable $T$.
\end{Corollary}

For any type $T=(t_1,\dots,t_k)$ there is a natural decomposition of the vector space $V=\bigoplus V_i^T$, where $V_i^T$ is  the vector space corresponding to the block $B_i$. That means $V_i^T$ is spanned by $e_{t_1+\dots+t_{i-1}+1},\dots,e_{t_1+\dots+t_{i}}$ and we have $\dim(V_i^T)=t_i$.

\section{Geometry of $X(C_n)$}\label{geometry}
In this section we will describe the intersection of the variety $X(C_n)$ with the exceptional divisors $E_r$. We will often use the representation of the points $P\in E_r$ as triples $(A_P,U_P,B_P)$, and as projections $P_i$, described in Section \ref{points-E_r}.

\subsection{Intersection of $X(C_n)$ with $E_{r}$ for $r<n-2$}

In this subsection we consider the case $1\le r< n-2$. We fix $(A,U)$ and describe the fiber over this pair, i.e. all possible $B$ such that the triple $(A,U,B)$ defines a point in the intersection $X(C_n)\cap E_r$. In Lemmas~\ref{lowrankequation} and \ref{lemma:XE1} we consider the case $r=1$. In Lemma~\ref{lowrankequation} we prove that $B$ lies in the total transform of $L(C_{n-1})$ viewed as the subspace of $V/U$. In Lemma~\ref{lemma:XE1}, we prove that $B$ also lies in the strict transform of $L(C_{n-1})$ and also the other inclusion, i.e. that every such $B$ gives us a point in $X(C_n)\cap E_1$. Then we use this result to prove in Lemma~\ref{lemma:recursion-XE_r} that the same happens in the case when $r>1$, namely the point $B$ lies in the strict transform of $L(C_{n-r})$.

In what follows we will be working with points $(A,U,B)\in X(C_n)\cap E_1$. The space $U$ must be the image of a rank one matrix from $L(C_n)$. By Corollary~\ref{lemma:blockform}, after acting with suitable $\sigma\in D_n$ is the space $U$ generated by a vector of the form $e_1+ae_2$. Thus, throughout this subsection we may, without loss of generality, assume that $U$ is generated by $e_1+ae_2$.

\begin{Lemma}\label{lowrankequation}
Let  $P\in X(C_n)\cap E_1$ be such that $U_P$ is generated by $e_1+ae_2$. By writing $B:=(B_P)_1\in\PP(\S^2(V/U_P))$ as a matrix in basis $\widehat{e_2},\dots,\widehat{e_n}$, we have $B_{kl}=0$ for $n-2>|k-l|>1$, i.e $B$ lies in the linear space for the smaller cycle.
\end{Lemma}

\begin{proof}
We index the rows and columns of the matrix $B$ naturally from $2$ to $n$, i.e. there is no row or column indexed by one. Let us fix $k,l$ as in the statement. Since $B$ is symmetric we may assume $k\neq n$. Then we distinguish two cases. Firstly, suppose $k,l\neq 2$. 

Note the corresponding basis of $(V/U_P)^*=U_P^\perp$, in which the matrix $B$ is written, is $-ae_1^*+e_2^*,e_3^*,\dots,e_n^*$.
From the description presented in Subsection~\ref{points-E_r}, we know that $$B_{kl}=(e_1^*\wedge e_k^*)^T P_2 (e_1^*\wedge e_l^*).$$

Now we consider any point $P'\in X(C_n)$ for which $P'_1$ is regular. Let  

$$M=\begin{pmatrix}
(e_1^*)^T\\
(e_k^*)^T
\end{pmatrix}
\cdot P'_1
\cdot
\begin{pmatrix}
e_1^*  e_l^*
\end{pmatrix}.$$

Then $$
(e_1^*\wedge e_k^*)^T P'_2 (e_1^*\wedge e_l^*)=\det(M).$$

Since the matrix $P'_1\in L(C_n)$ and $k\notin \{2,n\}$, the matrix $M$ has zeros in the last row. It follows that $\det(M)=0$.

This equation holds for all points $P'\in X(C_n)$ for which $P'_1$ is regular, therefore it must hold for all points from $X(C_n)$. Therefore, it holds also for the point $P$, and $B_{kl}=0$.

Now we consider the second case, when one of $k,l$ is equal to two, since $B$ is symmetric we may assume $l=2$.

In this case we have 
$$B_{kl}=(e_1^*\wedge e_k^*)^T P_2 (e_1^*\wedge (-ae_1^*+e_2^*)).$$

Analogously, as in the first case, we have $$(e_1^*\wedge e_k^*)^T P'_2 (e_1^*\wedge (-ae_1^*+e_2^*))=0,$$
for all regular matrices $P'\in X(C_n)$ and we obtain the desired conclusion, namely $B_{k2}=0$.
\end{proof}

\begin{Lemma} \label{lemma:XE1}
The variety $X(C_n)\cap E_1$ has $n$ components, each of them is isomorphic to an $X(C_{n-1})$-bundle over $\PP^1$, i.e.~the fiber over every point is isomorphic to $X(C_{n-1})$.
\end{Lemma}
\begin{proof}
Every point $P\in E_1$ can be represented as a pair $(U_P,B_P)$ where $U_P\in G(1,V)$ and $B_P\in \CQ(V/U_P)$. We do not need to specify $A_P$  because the space of complete quadrics over 1-dimensional space is just a point.

If $P\in X(C_n)\cap E_1 $, then $U_P$ is the image of a rank 1 matrix in $L(C_n)$.
By Lemma \ref{lemma:lowrank}, such a matrix is in a cyclic-block form, and therefore its image is of the type $[ae_i+be_{i+1}]$, for some $1\le i\le n$. This means that the projection of $X(C_n)\cap E_1$ to $G(1,V)$ consists of (at most) $n$ lines $\PP^1$.  We now describe the $n$ isomorphic components, each one mapping to a fixed $\PP^1$.

 We fix the space $U$, which is without loss of generality generated by $u=e_1+\alpha e_2$. 
Consider a point $B\in \CQ(V/U)$, such that $B$ is a regular matrix, written in the basis $\widehat{e_2},\widehat{e_3},\dots,\widehat{e_n}$ with rows and columns indexed from $2$ to $n$.  We know, by Lemma~\ref{lowrankequation}, that the necessary condition for $(U,B)\in X(C_n)$ is $B_{ij}=0$, for all $n-2>|i-j|>1$. We will show that this condition is also sufficient. For this we fix a regular matrix $B$ satisfying the condition and our goal is to prove that $(U,B)\in X(C_n)$.

First, assume $\alpha\neq 0$. Consider a matrix $C\in S^2(V)$, given as follows: $C_{ij}=B_{ij}$ for $i,j\ge 2,(i,j)\notin\{ (2,n),(n,2)\}$;
$C_{1n}=C_{n1}=-\frac{1}{\alpha}B_{2n}$ and the rest of the entries of $C$ are 0. Clearly, $C\in L(C_n)$ and when we restrict $C$ to an element of $S^2(V/U)$ we get $B$.
 
Let $A$ be the unique (up to scaling) symmetric matrix with image $U$. Then $A+\lambda C\in L(C_n)$ and $A+\lambda C$ converges to $(U,B)$ when $\lambda\rightarrow 0$, therefore $(U,B)\in X(C_n)$.

Now, assume $\alpha=0$. 
We construct a sequence of matrices in $L(C_n)$ which converges to $(U,B)$. In this case we may assume $A_{11}=1$ and other entries of $A$ are 0. We consider matrices $C(\lambda)$, given as follows:
\begin{itemize}
    \item $C(\lambda)_{ij}=\lambda^2 B_{ij}$ for all $i,j\ge 2; (i,j)\notin\{ (n,2),(2,n),(2,2),(n,n)\},$
    \item $C(\lambda)_{12}=C(\lambda)_{21}=-\lambda$,
    \item $C(\lambda){1n}=C(\lambda)_{n1}=\lambda B_{2n}$,
    \item $C(\lambda)_{22}=B_{22}+\lambda^2$,
    \item $C(\lambda)_{nn}=B_{nn}+\lambda^2 B_{2n}^2$,
    \item all other entries are zero.
\end{itemize}

One can check that also in this case the sequence $A+C(\lambda)$ converges to $(U,B)$ in $\CQ(V)$.

Therefore, in all cases we prove that $(U,B)\in X(C_n)$. That means that for the fixed one-dimensional subspace $U$, the points $(U,B)\in X(C_n)\cap E_1$, where $B$ is a regular matrix, are precisely those where the matrix $B$ belongs to the space for the cycle $L(C_{n-1})$ inside the space $V/U$.

We keep the space $U$ fixed and consider a point $(U,B)\in X(C_n)\cap E_1$ such that $B\in\CQ(V/U)$ is not a regular matrix. Since $(U,B)\in X(C_n)$, there exists a sequence of matrices $\{A_k\}$ such that $A_k\in L(C_n)$ and \{$A_k$\} converges to $(U,B)$. We may assume $(u^*)^TA_iu^*=1$ for all $i$, since $A_i$ converges to A in $\PP(S^2 (V))$. There is a rational map $\psi:\PP(S^2(V))\dashrightarrow \PP(S^2(V/U)$, which is a composition of natural maps $\PP(S^2(V))\dashrightarrow\PP(S^2(V\wedge V))$ and $\PP(S^2(V\wedge V))\dashrightarrow \PP(S^2(V/U)$. Note that this map is defined on all regular matrices $M$ with the property $(u^*)^TMu^*\neq 0$, and the  image of such a matrix is a regular matrix.

Now we consider the sequence of points in $\CQ(V)$ defined by $\{(U,\psi(A_k))\}$. Clearly, $(U,\psi(A_k))\in X(C_n)\cap E_1$. We claim that this sequence also converges to $(U,B)$. The coordinates in $\PP(S^2 (V))$ are constant - they are equal to the unique matrix $A$ with the image $U$. For all $k$, all non-zero coordinates in $\PP(S^2(\bigwedge^i V))$, i.e. those which correspond to coordinates in $\PP(S^2(\bigwedge^{i-1} (V/U)))$, are the same for $A_k$ and $(U,\psi(A_k))$ by the definition of $\psi$. Thus, on these coordinates the  sequence $\{(U,\psi(A_k))\}$ converges to the same point as the sequence $\{A_k\}$ which is the point $(U,B)$. To sum up, it converges in all coordinates, when we regard $\CQ(V)$ as being embedded in the product of projective spaces.

This means that the set of all points $(U,B)$ which are inside $X(C_n)\cap E_1$ is the closure of all such points where $B$ is a regular matrix. Moreover, we know that for regular matrix $B\in S^2(V/U)$ the point $(U,B)\in X(C_n)\cap E_1$ if and only if $B\in L(C_{n-1})$. This means that for any $B\in \CQ(V/U)$, the point $(U,B)\in X(C_n)\cap E_1$ if and only if $B\in X(C_{n-1})$ (in the appropriate basis). Thus, the fiber over $U$ is isomorphic to $X(C_{n-1})$.
\end{proof}

For $1\le i\le n$, let us denote by $\ell_i$ the line in $\PP(V)$ which is spanned by $e_i$ and $e_{i+1}$, where $e_{n+1}:=e_1$. We identify these lines as the edges of the graph $C_n$ living in $\PP(V)$. Indeed, when we draw them in $\PP(V)$ we get the picture of $C_n$. For an edge $e$ of $C_n$ we can also denote $\ell_e$ the corresponding line in $\PP(V)$.

For the description of the intersection of $X(C_n)$ with more exceptional divisors $E_r$ we will consider the subgraphs $\Gamma$ of $C_n$ which have $n$ vertices and less than $n-1$ edges. We denote by $E(\Gamma)$ the set of edges of $\Gamma$. Such graphs are unions of disjoint paths. For a graph $\Gamma$ we will denote by $\gamma_1,\dots,\gamma_k$ all positive lengths of these paths. Clearly, $\gamma_1+\dots+\gamma_k=|E(\Gamma)|$. We denote the set of all subgraphs of $C_n$ with $n$ vertices by $\Sigma_n$. We denote the subset of all such graphs which do not contain the edge $(1,n)$ by $\Sigma'_n$.
We will also work with the graphs which have their edges labeled by numbers $1,\dots,r$, where $r$ is the number of edges in the graph. We denote the set of all labeled subgraphs with $n$ vertices of $C_n$ by $\tilde{\Sigma}_n$. Also for a labeled graph $\tilde{\Gamma}\in \tilde{\Sigma}_n$, we denote by $\gamma_1,\dots,\gamma_k$ the lengths of the paths.
Moreover, we denote by $f_{\tilde{\Gamma}}:[r]\rightarrow E(C_n)$, the labeling function corresponding to a labeled graph $\tilde{\Gamma}$, i.e. for number $i$ it returns the edge with label $i$.

\begin{Corollary}\label{cor:E1..Er}
Let $r\le n-3$. Then  $X(C_n)\cap E_1\cap E_2\cap \dots \cap E_r$ has $n(n-1)\dots(n-r+1)$ components. Each component corresponds to a labeled graph $\tilde{\Gamma}\in\tilde{\Sigma}_n$ with $r$ edges.
We fix one of these components given by the graph $\tilde{\Gamma}$. Consider an open subset of this component which is formed by the points which do not lie in the other components. Then a point in this open subset is uniquely determined by the choice of $r$ points $p_1,\dots,p_r\in\PP(V)$, such that $p_i\in\ell_{f_{\tilde{\Gamma}}(i)}$, $p_i\neq e_j$ and $A\in X(C_{n-r})$, where $X(C_{n-r})\subset \CQ(V/(p_1,\dots,p_r))$.
\end{Corollary}
\begin{proof}
We prove it by induction on $r$. For $r=0$ there is nothing to prove. Assume that the statement holds for $r$. Then to get a point in $X(C_n)\cap E_1\cap E_2\cap \dots \cap E_{r+1}$ we need to have a graph $\tilde{\Gamma}\in\tilde{\Sigma}_n$, points $p_1,\dots,p_r\in\PP(V)$, such that $p_i\in\ell_{f_{\tilde{\Gamma}}(i)}$, $p_i\neq e_j$ and $A\in X(C_{n-r})\cap E'_1$, where $E'_1$ is the first exceptional divisor in $\CQ(V/(p_1,\dots,p_r))$. Now we use the Lemma~\ref{lemma:XE1} to see that $X(C_{n-r})\cap E'_1$ has $n-r$ components which naturally correspond to the edges of $C_n$ which were not chosen yet. Thus, to get the point here we need to choose one such edge and a point $p_{r+1}\in \ell$, where $\ell $ is the corresponding line in the quotient space $\PP(V/(p_1,\dots,p_r))$. However, the line $\ell$ is the image of line $\ell_{r+1}$ under the quotient map $\PP(V)\rightarrow \PP(V/(p_1,\dots,p_r))$. Thus, we can lift the point $p_{r+1}$ to the point in $\ell_{r+1}\subset \PP(V)$, since all $p_i$ are distinct from $e_j$.
\end{proof}

\begin{Lemma}\label{lemma:recursion-XE_r}
Let $1\le r \le n-3$. Let $U$ be an $r$-dimensional subspace which is the image of some matrix in $L(C_n)$ and $A\in \CQ(U)$. Then the set of all points $B\in \CQ(V/U)$ such that $(A,U,B)\in X(C_n)$ is a subset of $X(C_{n-r})\subset\CQ(V/U)$.
\end{Lemma}
\begin{proof}

We will prove it by induction on $r$. The case $r=1$ is solved in Lemma \ref{lemma:XE1}. Now we assume that the statement holds for all $r'<r$. 

Assume that $A$ is not a regular matrix. Then we look at the smallest $i$ such that $(A,U,B)\in E_i$, i.e. $i=\dim(\Im(A_1))$. Thus we can write $(A,U,B)$ as the element of $E_i$ in the form $(A_1,U',B')$, where $U'=\Im(A_1)$, and $B'\in \CQ(V/U')$. From the induction hypothesis we know that $B'$ lies in the space for the smaller cycle. Since $(A,U,B)\in E_r$, we know that $B'$ lies in the $(r-i)$-th exceptional divisor in $\CQ(V/U')$. Thus, we can write $B'$ in the form $(A',U,B)$.  We apply induction once again for the point $B'$ to get that $B$ is in the space for the smaller cycle.

Now we assume that $A$ is a regular matrix. We can also interpret $A$ as the element of $S^2(V)$, such that the rank of $A$ is $r$. By Lemma \ref{lemma:lowrank} the matrix $A$ is in a cyclic-block form. Thus in its image $U$ there exists a vector with just two consecutive non-zero coordinates, since it is the first vector of some block. Let us pick such vector and denote it by $u$. Without loss of generality we may assume $u=e_1+ae_2$.

Since $P:=(A,U,B)\in X(C_n)$ there must exist a sequence of regular matrices $\{A^k\}$ of $L(C_n)$ which converges to $(A,U,B)$. Here we are using $k$ as an upper index, we will not work with any powers of matrices. We will write these matrices in the basis $u,e_2,e_3,\dots,e_n$. In this basis we have $A_{11}\neq 0$, thus we may assume $A_{11}=1$. We know that $A^k$ converges to $A$ and since we are working projectively, we may also assume $A^k_{11}=1$ for all $k$. Note that in this basis we have $M\in L(C_n)$ if and only if $M_{ij}=0$, for all $n-1>|i-j|>1, (i,j)\notin \{(2,n),(n,2)\}$ and $M_{2n}=M_{n2}=-aM_{1n}=-aM_{n1}$.

Let us assume $a\neq 0$. Now we construct the sequence $\tilde{A}^k$ of regular matrices in $L(C_n)$ as follows:

\begin{itemize}
    \item $\tilde{A}^k_{11}=1$,
    \item $\tilde{A}^k_{22}=\lambda_k(A^k_{22}-(A^k_{12})^2)$,
    \item $\tilde{A}^k_{nn}=\lambda_k(A^k_{nn}-(A^k_{1n})^2)$,
    \item $\tilde{A}^k_{2n}=\tilde{A}^k_{n2}=\lambda_k(A^k_{2n}-A^k_{12}A^k_{1n})$,
    \item $\tilde{A}^k_{1n}=\tilde{A}^k_{n1}=-\frac 1a \tilde{A}^k_{2n}$
    \item $\tilde{A}^k_{ij}=\lambda_kA^k_{ij}$ for other $(i,j)$.
\end{itemize}

where $\lambda_k$ are sufficiently small and $\{\lambda_k\}$ converges to 0. We claim that the limit $S\in \CQ(V)$ of this sequence satisfies $\rank(S_1)=1$ and image of $S_1$ is generated by $u$. This is clear since the only entry of $\tilde{A}^k$ which is not multiple of $\lambda_k$ is $\tilde{A}^k_{11}$. Thus, it is the only nonzero entry of $S_1$. This means that $S\in E_1$.

Moreover, we claim that $S_j=P_j$ for $j\ge r$. To show this, we will compare $j\times j$ minors of matrices $A^k$ and $\tilde A^k$. Since the image of $P_1$ and $S_1$ contains $u$, we only have to consider the minors which contain both first row and column.

We start with the matrices $\tilde A^k$. We note that the only products of $j$ entries of $\tilde A^k$ which do not contain $\lambda_k^j$ are those which contain the entry $\tilde A^k_{11}$. In these products appears only $\lambda_k^{j-1}$. Thus, in the limit survive only products which contain $\tilde A^k_{11}$. Thus the limit of $j\times j$ minors of the sequence $\{\tilde A_k\}$ is just the limit of $(j-1)\times(j-1)$ minors of submatrices which we obtain by deleting the first row and column.

Next we consider the matrices $A^k$. Since we are considering minors with the first row and column we may add a multiple of the first row (column) to any other row (column) and this will not change these minors. Thus we will do such operations to kill the entries $A^k_{12},A^k_{1n},A^k_{21},A^k_{n1}$. This means we perform the following operations: we subtract $A^k_{12}$-multiple of first row (column) to the second row (column) and $A^k_{1n}$-multiple of the first row (column) to the last row (column).

Let us denote the resulting matrix by $A'^k$. It will look as follows:

\begin{itemize}
    \item $A'^k_{12}=A'^k_{21}=A'^k_{1n}=A'^k_{n1}=0$,
    \item $A'^k_{22}=A^k_{22}-(A^k_{12})^2$,
    \item $A'^k_{nn}=A^k_{nn}-(A^k_{1n})^2$,
    \item $A'^k_{2n}=A'^k_{n2}=A^k_{2n}-A^k_{12}A^k_{1n}$,
    \item $A'^k_{ij}=A^k_{ij}$ for other $(i,j)$.
\end{itemize}

All $j\times j$ minors containing both first row and column of this matrix are equal to the $(j-1)\times (j-1)$ minors of the matrix which we obtain by deleting first row and column.

This means that in both sequences, $\{A^k\}$ and $\{\tilde A^k\}$, we are computing the limit of the same minors, only in $\tilde A^k$ all are multiplied by $\lambda_k^{j-1}$. Thus, $S_j=P_j$ for all $j\ge r$.

In the case $a=0$, this construction fails but we can still construct a sequence $\{\tilde A^k\}$ as follows:

\begin{itemize}
    \item $\tilde{A}^k_{11}=1$,
    \item $\tilde{A}^k_{1i}=\tilde{A}^k_{i1}=\lambda_k(A^k_{1i})$ for all $2\le i\le n$,
    \item $\tilde{A}^k_{ij}=\lambda_k^2A^k_{ij}$ for all $2\le i,j,\le n$.
\end{itemize}

In a similar way, one can prove that, also in this case, the limit $S$ of this sequence satisfies $S_j=P_j$ for $j\ge r$ and $S_1$ has rank one and image $[u]$, but we omit the technical details.

Since $S\in E_1$, we can write $S=(A_S,[u],B_S)$, where $B_S\in X(C_{n-1})\subset\CQ(V/[u])$. Moreover, since $S_j=P_j$ for all $j\ge r$ we can write $B_S=(A_{B_S},U,B)$. From induction hypothesis for $B_S$ it follows that $B$ is in the space for smaller cycle which concludes the lemma.

\end{proof}
\begin{Remark}
We believe that in a similar way one can show that in fact in Lemma \ref{lemma:recursion-XE_r} equality holds, i.e.~that the fiber is equal to the space for smaller cycle. However, this is even more technical. As we do not need this for our main result, we omit the proof for the sake of compactness.
\end{Remark}

\subsection{Intersection of $X(C_n)$ with $E_{n-2}$}

\begin{Lemma}\label{n-2_uniqueA}
 Let $U\subset V$ be an $(n-2)$-dimensional subspace such that there is no non-zero vector with only two non-zero coordinates in $U$. Then there exists a unique matrix (up to a multiple) $A\in L(C_n)$ such that $\Im A=U$.
\end{Lemma}

\begin{proof}
Let $a_1,\dots,a_n$ be the rows of $A$. We look at $a_1$. We know that $a_1\in U$ and $a_{1,3}=a_{1,4}=\dots=a_{1,n-1}=0$. This gives us together $n-1$ linear conditions on $a_1$. Thus there is at least one non-zero $a_1$ which satisfies all the conditions. If there were two linearly independent possibilities for $a_1$ then their suitable linear combination would be a vector in $U$ with only two non-zero coordinates, which is not possible. Thus, $a_1$ is uniquely determined up to a scalar multiple, and moreover $a_{1,1},a_{1,2},a_{1,n}$ are non-zero. After scaling, we may assume $a_{1,1}=1$. Analogously, $a_2$ is uniquely determined up to a scalar multiple, but moreover we want $a_{1,2}=a_{2,1}$, so we can pick $a_2$ such that this holds. We continue with $a_3,\dots, a_n$. In this way we uniquely construct $A$ which image lies in $U$. Moreover the minor of $A$ obtained by deleting first two columns and last two rows is equal to $a_{2,3}a_{3,4}\dots a_{n-2,n-1}a_{1,n}\neq 0$, thus the rank of $A$ is at least $n-2$. Thus, $\Im A=U$. To prove that $A\in L(C_n)$ we have to check only one more condition, namely $a_{1,n}=a_{n,1}$ which was not used during the construction. However, both $(n-1)\times(n-1)$ minors of $A$ obtained by deleting the first row and the last column, or the last row and the first column are 0, since $A$ is a rank $n-2$ matrix. If $a_{1,n}\neq a_{n,1}$ then these two minors are different which is a contradiction.
\end{proof}

\begin{Lemma}\label{n-2_anyB}

 Let $U\subset V$ be an $(n-2)$-dimensional subspace such that there is no non-zero vector with only two non-zero coordinates in $U$. Let $A\in L(C_n)$ be the unique matrix with $\Im A=U $. Let $B\in \CQ(V/U)=\PP(S^2 (V/U))$. If a point $P\in \CQ(V)$ is given by $(A_P,U_P,B_P)=(A,U,B)$, then $P\in X(C_n)$.
\end{Lemma}

\begin{proof}
We write $B$ in the basis $\widehat{e_1}, \widehat{e_2}$. We denote by $C$ the matrix such that $C_{ij}=B_{ij}$ for $1\le i,j\le 2$ and $C_{ij}=0$ otherwise. Assume that $B_P$ is regular. We claim that $A+\lambda C$ is regular for all $\lambda\neq 0$.

Indeed, suppose that some linear combination of rows of $A+\lambda C$ is zero. If this linear combination involves nontrivially the first or second row, then by taking the same combination in $A$ we obtain a vector in $\Im A$ which has only two non-zero coordinates. If not, it is just a linear combination of the last $n-2$ rows of $A$. However, these can be dependent only if one of the numbers $A_{23},A_{34},\dots,A_{(n-1)n}$ is 0, which again implies that there is a vector with just two non-zero coordinates in $\Im A$, a contradiction.

Therefore, the matrix $A+\lambda C$ defines a unique point $P(\lambda):=\phi(A+\lambda C)$ in $\CQ(V)$. We can parametrize all coordinates of $P(\lambda)$ by $\lambda$. Since $X(C_n)$ is closed, the point $P(0)$, which is obtained by putting $\lambda=0$, is also in $X(C_n)$. It is straightforward to check that $P(0)$ corresponds to $(A,U,B)$, which shows that $(A,U,B)\in X(C_n)$.

If $B$ is not regular, then it has rank 1. In this case $A+\lambda C$ has rank $n-1$ and the proof follows in the same way, since the map $\phi$ from Definition~$\ref{CQ1}$ is actually well defined also for rank $n-1$ matrices, i.e. a rank $n-1$ matrix also determines a unique point in the space of complete quadrics.
\end{proof}

Now we sum up Lemmas~\ref{n-2_uniqueA} and \ref{n-2_anyB}, and we show the most important result of this subsection. We will say that a (possibly rational) map $f:X\rightarrow Y$ is an isomorphism on an open subset of $Y$, if there exists an open subset $U\subset Y$ such that $f:f^{-1}(U)\rightarrow U$ is an isomorphism.

\begin{Lemma}\label{lemma:En-2}
Let $\mathcal Q$ be the quotient bundle over Grassmannian $G(n-2,V)$. Let $\psi: X(C_n)\cap E_{n-2} \rightarrow S^2(\mathcal Q) $ be the map which sends $P=(A_P,U_P,B_P)$ to $(U_P,B_P)$ in $ S^2(\mathcal Q)$. Then this map is an isomorphism on an open subset of $S^2(\mathcal Q)$.
\end{Lemma}
\begin{proof}
 We can define a rational map $\overline{\psi}: S^2(\mathcal Q)\dashrightarrow X(C_n)\cap E_{n-2}$, which sends $(U,B)$ to the unique point $(A,U,B)\in X(C_n)\cap E_{n-2}$ by Lemmas~\ref{n-2_uniqueA} and \ref{n-2_anyB}. The uniqueness of $A$ is guaranteed only on the open subset of $S^2(\mathcal Q)$ where the space $U$ does not contain a vector with only two non-zero coordinates, so it is only a rational map.
 
 Note that for $U$ which does not contain a vector with two non-zero coordinates, there does not exist any point $(A',U,B)\in X(C_n)\cap E_{n-2}$ where $A'$ is not a regular matrix. Indeed, if it exists, then $A'_1$, interpreted as the element of $S^2(V)$, has rank at most $n-3$. By Lemma \ref{lemma:lowrank}, it is in the cyclic-block form and therefore its image contains a vector with just two non-zero coordinates which is a contradiction.

Thus, the maps $\psi$ and $\overline{\psi}$ are inverses to each other on an open set of $S^2(\mathcal Q)$, which proves the lemma.
\end{proof}

\subsection{Intersection of $X(C_n)$ with more exceptional divisors $E_i$}

For describing such intersections, we will work with the subvarieties of a Grassmannian, more precisely with their classes in the Chow ring of Grassmannian. For a partition $\lambda$ we denote by $s_\lambda$ the corresponding Schur polynomial, \cite{macdonald}.  Recall that the Chow ring of $G(n,r)$ is generated by $s_\lambda$ where $\lambda$ fits inside an $r\times (n-r)$ rectangle.


\begin{Lemma}\label{lemma:XE1..ErEn-2}
Let $r\le n-3$. Then $X(C_n)\cap E_1\cap E_2\cap \dots \cap E_r\cap E_{n-2}$ has $n(n-1)\dots (n-r+1)$ components, each of them corresponding to a labeled graph $\tilde{\Gamma}\in \tilde{\Sigma}_n$ with $r$ edges. Moreover, for every component $Y_{\tilde{\Gamma}}$ there exists a closed subvariety $G_{\tilde{\Gamma}}$ of Grassmannian $G(n-2,V)$ and the quotient bundle $\mathcal Q_{\tilde{\Gamma}}$ over $G_{\tilde{\Gamma}}$ such that the map $\psi: Y_{\tilde{\Gamma}} \dashrightarrow S^2(\mathcal Q_{\tilde\Gamma}) $, which sends $P=(A_P,U_P,B_P)$ to $(U_P,B_P)$ in $ S^2(\mathcal Q_{\tilde{\Gamma}})$, is an isomorphism on an open subset of $S^2(\mathcal Q_{\tilde{\Gamma}})$. Moreover, if $\gamma_1,\dots,\gamma_k$ denotes the lengths of paths in $\tilde{\Gamma}$, then the class of $G_{\tilde{\Gamma}}$ in the Chow ring of $G(n-2,V)$ is $s_{1,\dots,1}\dots s_{1,\dots,1}$, where the numbers of ones in each class are $\gamma_1,\dots,\gamma_k$.
\end{Lemma}
\begin{proof}
By Corollary~\ref{cor:E1..Er}, $X(C_n)\cap E_1\cap E_2\cap \dots \cap E_r$ has $n(n-1)\dots (n-r+1)$ components given by the graphs $\tilde{\Gamma}\in\tilde{\Sigma}_n$ with $r$ edges. By intersecting each of these components with $E_{n-2}$ we get $n(n-1)\dots(n-r+1)$ components of $X(C_n)\cap E_1\cap E_2\cap \dots \cap E_r\cap E_{n-2}$. We fix a graph $\tilde{\Gamma}$.
Then consider the set $$G_{\tilde{\Gamma}}=\overline{\{ \Lambda\in G(n-2,V); \forall i\le r: \Lambda\cap \ell_{f_{\tilde{\Gamma}}(i)}\neq\emptyset,\forall j\le n: e_j\notin \Lambda\}}.$$ We will prove that it satisfies the conditions from the statement of our lemma. We still have the map  $\psi:Y_{\tilde{\Gamma}} \rightarrow S^2(\mathcal Q)$, since it is the restriction of a map defined on $E_{n-2}$. Consider the projection $\xi:S^2(\mathcal Q)\rightarrow G(n-2,V)$. We claim that the image of $\xi\circ\psi$ map lies in $G_{\tilde{\Gamma}}$. 

Consider any $y\in Y_{\tilde{\Gamma}}$ which does not lie in other component. By Corollary~\ref{cor:E1..Er}, the point $y$ is defined by points $p_1,\dots,p_r$ such that $p_i\in \ell_{f_{\tilde{\Gamma}}(i)}$. The space $\xi(\psi(y))$ contains points $p_1,\dots,p_r$, thus $\xi(\psi(y))\in G_{\tilde{\Gamma}}$. Consequently, the image of $\psi$ lies in $S^2(\mathcal Q_{\tilde{\Gamma}})$, hence we have a map 
$\psi:Y_{\tilde{\Gamma}} \rightarrow S^2(\mathcal  Q_{\tilde{\Gamma}})$. Moreover, we can define the map $$\overline{\psi}: S^2(\mathcal Q_{\tilde{\Gamma}})\dashrightarrow Y_{\tilde{\Gamma}}$$
as follows: consider a point $(U,B)\in S^2(\mathcal Q_{\tilde{\Gamma}}) $, where $U\in G_{\tilde{\Gamma}}$ and $B\in \PP(S^2(V/U))$. We define the points $p_i:=U\cap \ell_{f_{\tilde{\Gamma}}(i)}$.

Let $\mathcal Q'$ be the quotient bundle over Grassmannian $G(n-r-2,V/(p_1,\dots,p_r))$. We have a natural embedding $i:G(n-r-2,V/(p_1,\dots,p_r))\rightarrow G_{\tilde{\Gamma}}$ which sends the space $\Lambda$ to its preimage in the projection $V\rightarrow V/(p_1,\dots,p_r)$. Consequently, we have also an embedding $\tilde{i}:S^2(\mathcal Q')\rightarrow S^2(\mathcal Q_{\tilde{\Gamma}})$. Now we apply Lemma~\ref{lemma:En-2} to the space $\CQ(V/(p_1,\dots,p_r))$. The map $\psi':X(C_{n-r})\cap E_{n-r-2}\rightarrow S^2(\mathcal Q') $ is an isomorphism on an open subset and it has its inverse $\overline{\psi'}:S^2(\mathcal Q') \dashrightarrow X(C_{n-r})\cap E_{n-r-2}$.

We define $A\in X(C_{n-r})\subset\CQ(V/(p_1,\dots,p_r))$ as $A=\overline{\psi'}(i^{-1}(U,B))$. Now we define $\overline{\psi}(U,B)$ as the point in $Y_{\tilde{\Gamma}}$ determined by $p_1,\dots,p_r$ and $A$ as in Corollary~\ref{cor:E1..Er}. 

The map $\overline{\psi}$ is only a rational map, since it is well-defined only if the intersection of $U$ with all $\ell_{f_{
\tilde{\Gamma}}(i)}$ is just a point different from $e_j$ and if $i^{-1}(U,B)$ belongs to the subset of $S^2(\mathcal Q')$ where the rational map $\overline{\psi'}$ is defined.

By construction, the maps $\psi$ and $\overline{\psi}$ are inverses of each other on an open subset of $S^2(\mathcal Q_{\tilde{\Gamma}})$.
To conclude the lemma, we need to determine the class of $G_{\tilde{\Gamma}}$ in the Chow ring of Grassmannian. A linear space $\Lambda\subset\PP(V)$ intersects all of the lines $\ell_1,\dots \ell_k$ in points different from $e_j$ if and only if it intersects the $k$-plane spanned by  $\ell_1,\dots \ell_k$ in a $(k-1)$-dimensional space and does not contain points $e_j$. The closure of the set of all such spaces $\Lambda$ is the set of all spaces $\Lambda$ which intersect the $k$-plane spanned by  $\ell_1,\dots \ell_k$ in a $(k-1)$-dimensional space. This is, indeed, true for any $k$ consecutive lines. This is exactly the condition for the class $s_{1,\dots,1}$ with $k$ ones.  From this it follows that the class of $G_{\tilde{\Gamma}}$ is exactly the class from the statement of our lemma.
 \end{proof}
 
We finish this section with two technical lemmas which will allow us to compute the intersection products in  Section~\ref{section:proof}. 

\begin{Lemma}\label{lemma:argument}
Let $\phi_r:E_r\rightarrow G(r,V)$ be the natural map which sends the point $p\in E_r$ to the image of $\pi_r(p)\in \PP(S^2(\bigwedge^r V))$. Let $\Lambda$ be the $(n-r-1)$-space in $\PP(V)$. Let $\Omega(\Lambda)=\{\Theta\in G(r,V): \Theta\cap\Lambda\neq\emptyset\}\subset G(r,V)$. Then $2[\phi_r^{-1}(\Omega(\Lambda))]=S_rL_r$.

\end{Lemma} 
\begin{proof}
We have the following commutative diagram:
\[
\begin{tikzcd}                      
                     E_r\arrow[r,"\phi_r"]\arrow[rrd,"\tilde{\pi}_r"]&G(r,V)\arrow[r,hook,"i"]& \PP(\bigwedge^r(V))\arrow[d,"\nu_2"] \\
                     &&\PP(S^2(\bigwedge^r(V)))
\end{tikzcd}
\]
where $G(r,V)$ is embedded in $\PP(\bigwedge^r(V))$ by Pl\"ucker coordinates, $\nu_2$ is the second Veronese map and $\tilde{\pi}_r$ is the restriction of $\pi_r$ to $E_r$.
Let $H$ be a general hyperplane in  $\PP(S^2(\bigwedge^r(V)))$. From the diagram we have the following equality in the Chow ring of $E_r$: $$[\tilde{\pi}_r^{-1}(H)]=[\phi_r^{-1}(i^{-1}(\nu_2^{-1}(H)))].$$

Let $H'$ be a general hyperplane in $\PP(\bigwedge^r(V))$. Since $\nu_2$ is the second Veronese map, we get $[\nu_2^{-1}(H)]=2[H']$. Moreover we know that $[i^{-1}(H')]=[\Omega(\Lambda)]$. Thus, $$[\tilde{\pi}_r^{-1}(H)]=2[\phi_r^{-1}(\Omega(\Lambda))].$$

This equality also holds in the Chow ring of $\CQ(V)$. We have $\tilde{\pi}_r^{-1}(H)=\pi_r^{-1}(H)\cap E_r$, thus $$L_rS_r=[\tilde{\pi}_r^{-1}(H)]=2[\phi_r^{-1}(\Omega(\Lambda))].$$

\end{proof}

\begin{Lemma}\label{lemma:puzzle}
Consider the component $Y_{\tilde{\Gamma}}$ of $X(C_n)\cap E_1\cap\dots\cap E_r$ given by the graph $\tilde{\Gamma}\in \tilde{\Sigma}_n$ with $r$ edges. Let $\mathcal E_{\tilde{\Gamma}}$ be the set of edges of $\tilde{\Gamma}$ such that the right end of the edge has degree one in ${\tilde{\Gamma}}$. (Here, the right end of edge $(1,n)$ is the vertex $1$ and otherwise it is the vertex with higher number.)

For each $e\in \mathcal E_{\tilde{\Gamma}}$, let $q_e$ be a general point on the line $\ell_e$. Let $\Lambda_{\tilde{\Gamma}}=\spann(\{q_e \:\; e\in\mathcal E_{\tilde{\Gamma}}\}\cup\{e_i \:\; i \text{ is isolated vertex in }\tilde{\Gamma}\})$. Then the intersection $Y_{\tilde{\Gamma}}\cap \phi_r^{-1}(\Omega(\Lambda_{\tilde{\Gamma}}))$ is transverse and $Y_{\tilde{\Gamma}}\cap \phi_r^{-1}(\Omega(\Lambda_{\tilde{\Gamma}}))$ has $|\mathcal E_{\tilde{\Gamma}}|$ components. Each of them is birational to the component of $X(C_{n-1})\cap E_1\cap\dots\cap E_{r-1}$ given by the graph $\tilde{\Gamma}/e$ for some $e\in \mathcal E_{\tilde{\Gamma}}$. In the graph $\tilde{\Gamma}/e$ we decrease the labels of every edge with higher label than $e$ by one, so now the graph $\tilde{\Gamma}/e$ has the edges labeled by $1,\dots,r-1$. Note that the graph $\tilde{\Gamma}/e$ is uniquely determined as the subgraph of $C_{n-1}$ up to a possible rotation.
\end{Lemma}

\begin{proof}
By Corollary~\ref{cor:E1..Er}, we know that a general point in $P\in Y_{\tilde{\Gamma}}$ is given by points $p_1,\dots,p_r$ such that $p_i\in\ell_{f_{\tilde{\Gamma}}(i)}$ and $A\in CQ(X_{n-r})$. Clearly, $\phi_r(P)=\spann\{p_1,\dots,p_r\}$. 

Consider the matrix with rows $p_1,\dots,p_r,q_e,e_i$ for $e\in \mathcal E_{\tilde{\Gamma}}$, and all $i$ such that $i$ is an isolated vertex in $\tilde{\Gamma}$. The spaces $\spann\{p_1,\dots,p_r\}$ and $\Lambda_{\tilde{\Gamma}}$ intersect if and only if this matrix is singular. We can change the order of rows and split the matrix into blocks, such that every block corresponds to a connected component of $\tilde{\Gamma}$. From this follows that this matrix is singular if and only if we have that $q_e\in \spann\{p_1,\dots,p_r\}$ for some edge $e$. This happens if and only if $p_i=q_e$ for some $1\le i\le r$ and $e\in \mathcal E_{\tilde{\Gamma}}$.
 
 Thus, we can see that $Y_{\tilde{\Gamma}}\cap \phi_r^{-1}(\Omega(\Lambda))$ has $|\mathcal E_{\tilde{\Gamma}}|$ components depending on $e$. Let us fix one of these components by fixing the edge $e$. We claim that it is birational to the component of $X(C_{n-1})\cap E_1\dots\cap E_{r-1}$ given by $\tilde{\Gamma}/e$. Indeed, we will realise $L(C_{n-1})\subset \PP(V/(q_e))$. Then the birational map is trivial, it sends a point corresponding to $p_1,\dots,p_r,A$ to the point corresponding to $\widehat{p_1},\dots,\widehat{p_r},A$ when the point $p_i$, such that $p_i=q_e$ is ignored.
 
 It remains to prove that the intersection $Y_{\tilde{\Gamma}}\cap \phi_r^{-1}(\Omega(\Lambda_{\tilde{\Gamma}}))$ is transverse. It is sufficient to show the transversality after projecting with $\phi_r$ to the Grassmannian $G(r,V)$. Note that $\phi_r^{-1}(\Omega(\Lambda_{\tilde{\Gamma}}))$ is a divisor in $S_r$, thus the intersection is not transverse only if, at some point of the intersection, the tangent space to $Y_{\tilde{\Gamma}}$ is a subspace of the tangent space to $\phi_r^{-1}(\Omega(\Lambda_{\tilde{\Gamma}}))$. However, if we show that this does not happen after projecting to $G(r,V)$ it can not happen before projecting.
 
Notice that $$\phi_r(Y_{\tilde{\Gamma}})=\overline{\{\Theta\in G(r,V); \forall e\in \mathcal E(\tilde{\Gamma}): \Theta\cap \ell_e\neq\emptyset, \forall j\le n: e_j\notin \Theta\}}.$$

Consider a point $p\in \phi_r(Y_{\tilde{\Gamma}})\cap \Omega(\Lambda_{\tilde{\Gamma}})$. We know that $p$ is spanned by the points $p_1,\dots,p_r$ such that $p_i\in \ell_{f_{\tilde{\Gamma}}(i)}$ and $p_i=q_e$ for some $1\le i\le r$, $e\in \mathcal E_{\tilde{\Gamma}}$. Without loss of generality we may assume that $i=f(i)=1$ and $e=(1,2)$. We will think about the points $p_i$ as a vector in $V$. Let $v_1,\dots,v_{n-r}$ be the vectors such that $p_1,\dots,p_r,v_1,\dots,v_{n-r}$ span $V$. The tangent space to Grassmannian $G(r,V)$ at point $p$ is generated by vectors $p_i\otimes v_j $ for $1\le i\le r, 1\le j\le n-r$.

The tangent space to $\Omega(\Lambda_{\tilde{\Gamma}})$ is spanned by $p_i\otimes v_j $ for $2\le i\le r, 1\le j\le n-r$ and $p_1\otimes q_e$, $p_1\otimes e_i$ such that $(1,2)\neq e\in\mathcal E_{\tilde{\Gamma}}$ and $i$ is isolated vertex in $\tilde{\Gamma}$. This is true since we want to move in the direction such that our space will still intersect $\Lambda_{\tilde{\Gamma}}$. Thus, we can either not change vector $p_1$ or change it in the direction of $\Lambda_{\tilde{\Gamma}}$.

The tangent space to $\phi_r(Y_{\tilde{\Gamma}})$ at  point $p$ is generated by the vectors $p_i\otimes e_j$ where $j$ is the left end of the edge $f_{\tilde{\Gamma}}(i)$. This is clear since any point in $\phi_r(Y_{\tilde{\Gamma}})$ is defined by $r$ points on lines $\ell_{f_{\tilde{\Gamma}}}(i)$.

The vector $p_1\otimes e_1$ does not lie in the tangent space of $\Omega(\Lambda_{\tilde{\Gamma}})$, which means that the tangent vectors generate the whole tangent space in $p$ to the Grassmannian $G(r,V)$, and the intersection is transverse.
\end{proof}

\section{Proof of Theorem~\ref{variety}}\label{section:proof}
In this section we will prove formulas for various intersections in the space of complete quadrics. Ultimately we also compute the intersection which proves the Theorem~\ref{variety}.

\begin{Lemma}\label{lemma:Pataki} We have
$S_{n-1}L_{n-1}^{n}=0.$
\end{Lemma}
\begin{proof}
Since $E_{n-1}$ projects by $\pi_{n-1}$ to rank one matrices and the dimension of the space of rank one matrices is $n-1$, it follows that $S_{n-1} \cdot L_{n-1}^{n}=0$, thus the statement holds, see also \cite[Proposition 5]{nie}.
\end{proof}

\begin{Lemma}\label{lemma:0}
Let $n-2>r>r'+1$. Then
$[X(C_n)]S_1S_2\dots S_{r'}S_{r}L_{n-1}^{2n-2-r'}=0.$
\end{Lemma}
\begin{proof}
Assume that $r'=0$. Then 
 consider the projection \[\pi:\CQ(V)\rightarrow \PP\left( S^2\left(\bigwedge^{r} V \right) \right) \times\ldots\times \PP \left(S^2\left(\bigwedge^{n-1} V\right) \right).\]

We show that $\dim(\pi(X(C_n) \cap E_r))\le 2n-r-1$. This will imply that, if we intersect with $2n-2$ general hyperplanes, the intersection is empty. Therefore, the intersection was empty also before the projection, hence $[X(C_n)]\cdot S_r \cdot L_{n-1}^{2n-2}=0$. It remains to compute $\dim(\pi(X(C_n)\cap E_r))$.

The projection $\pi(P)$ of a point $P\in E_{r'}$ is given by $(U_P,B_P)$, where $U_P$ s the image of rank $r$ matrix in $L(C_n)$. By Corollary~\ref{lemma:blockform}, there is a decomposition of $V=\bigoplus V_i^T$ and decomposition $U_P=\bigoplus (U_P)_i^T$, such that $(U_P)_i^T$ is a codimension one subspace of $V_i^T$. If we fix the type $T$, then such spaces $U_P$ form a subvariety of $G(r,V)$ of dimension 
$$
(\dim(V_1)-1)+(\dim(V_2)-1)+\dots+(\dim(V_{n-r})-1)=n-(n-r)=r.
$$
Since there are only finitely many types, the dimension of $\pi(E_r)$ after projecting to $G(r,V)$ is at most $r$. Now we compute the dimension of the fiber of $\pi(S_r)$ over fixed $U_P\in G(r,V)$. By Lemma~\ref{lemma:recursion-XE_r}, this fiber is the subset of the space $X(C_{n-r})\subset ´\CQ(V/U_P)$, thus its dimension is at most $2(n-r)-1$.
Therefore, we obtain that $$\dim(\pi(X(C_n)\cap E_r))\le r+2(n-r)-1=2n-r-1,$$ which implies the desired conclusion.

In the case $r'>0$, we again consider the projection \[\pi:\CQ(V)\rightarrow \PP\left( S^2\left(\bigwedge^{r} V \right) \right) \times\ldots\times \PP \left(S^2\left(\bigwedge^{n-1} V\right) \right),\]

and show again that $\dim(\pi(X(C_n)\cap E_1\cap\dots\cap E_{r'} \cap E_r)\le 2n-r-1<2n-2-r'$, and conclude in the same way.

From Corollary \ref{cor:E1..Er} it follows that a general point of $X(C_n)\cap E_1\cap\dots\cap E_{r'}$ is given by the points $p_1,\dots,p_{r'}$ and $A\in X(C_{n-r'})$. If we fix $p_1,\dots,p_{r'}$ and project we know from the proof for $r'=0$ that the dimension is at most $2(n-r')-(r-r')-1=2n-r-r'-1$. If we fix the component of $X(C_n)\cap E_1\cap\dots\cap E_{r'}$, then each $p_i$ lies on a fixed line, thus by considering also the points $p_1,\dots,p_{r'}$, the dimension of the image of $\pi(X(C_n)\cap E_1\cap\dots\cap E_{r'})$ can increase by at most $r'$.

Thus $\dim(\pi(X(C_n)\cap E_1\cap\dots\cap E_{r'} \cap E_r))\le 2n-r-r'-1+r'=2n-r-1$.

\end{proof}

\begin{Lemma}\label{lemma:killL_i}
Let $r\le n-3$. Then $$[X(C_n)]S_1S_2\dots S_rL_rL_{n-1}^{2n-2-r}=\frac{2nr(n-r)}{n-1}[X(C_{n-1})]S_1S_2\dots S_{r-1}L_{n-2}^{2n-2-r}.$$
\end{Lemma}
\begin{proof}
We know that $X(C_n)\cap E_1\cap\dots\cap E_r$ has $n(n-1)\dots(n-r+1)$ components, each of them is given by the graph $\tilde{\Gamma}\in\tilde{\Sigma}_n$ with $r$ edges. We denote this component by $Y_{\tilde{\Gamma}}$. By Lemma~\ref{lemma:argument}, we can intersect with $2[\phi^{-1}(\Omega(\Lambda))]$ instead of $S_rL_r$. Moreover, by Lemma~\ref{lemma:puzzle}, we can choose $\Lambda_{\tilde{\Gamma}}$, since then the corresponding variety intersects transversally with $Y_{\tilde{\Gamma}}$. Then we use Lemma~\ref{lemma:puzzle} to obtain
\begin{align*}[X(C_n)]S_1S_2\dots S_rL_rL_{n-1}^{2n-2-r}&=\sum_{\substack{\tilde{\Gamma}\in\tilde{\Sigma}_n\\|E(\tilde{\Gamma})|=r}} 2[Y_{\tilde{\Gamma}}\cap \phi^{-1}(\Omega(\Lambda_{\tilde{\Gamma}}))]L_{n-1}^{2n-2-r}\\
&=\sum_{\substack{\tilde{\Gamma}\in\tilde{\Sigma}_n\\|E(\tilde{\Gamma})|=r}}\sum_{e\in \mathcal E_{\tilde{\Gamma}}}2 [Y_{\tilde{\Gamma}/e}]L_{n-2}^{2n-2-r}.
\end{align*}
Now we would like to change this sum and sum trough all graphs with $r-1$ edges in $\tilde{\Sigma}_{n-1}$. We fix a graph $\tilde{\Upsilon}\in \tilde{\Sigma}_{n-1}$ and look how many pairs $(\tilde{\Gamma},e)$, where $e\in \mathcal E_{\tilde{\Gamma}}$ are there such that $\tilde{\Upsilon}=\tilde{\Gamma}/e$.

Firstly, it is clear that the edge $e$ can have any label from $1,\dots,r$ and if we change its label in $\tilde{\Gamma}$ to $r$ and decrease all higher labels by one we still get the same graph $\tilde{\Gamma}/e$. Thus, we may assume that $e$ has a label $r$ and multiply the number we get by $r$.

 Now, given a graph $\tilde{\Upsilon}$, consider its $n-1$ rotations, i.e. the graphs we obtain after acting with all elements in $\ZZ_n\subset\Aut(C_{n-1})$. Obviously, due to the labeling, all these graphs are different. We will compute the pairs $(\tilde{\Gamma},e)$ from which we obtain one of these graphs and $e$ has label $r$. We have to proceed in this way, since the graph $\tilde{\Gamma}/e$ is uniquely defined as the subgraph of $C_{n-1}$ only up to a rotation. We create the graph $\tilde{\Gamma}$ by replacing one of vertices of $\tilde{\Upsilon}$ by an edge. However we can use only vertices which are not left ends of an edge in $\tilde{\Upsilon}$, so we have $(n-1)-(r-1)=n-r$ vertices to choose from. The resulting graph is once again uniquely defined as the subgraph of $C_n$, only up to rotation. Thus, from all $n-1$ rotations of the graph $\tilde{\Upsilon}$ by replacing one vertex with an edge we can obtain all $n$ possible rotations of the resulting graph.
 
 Therefore, for any $\tilde{\Upsilon}\in\tilde{\Sigma}_{n-1}$ we get one of its $n-1$ rotations from $nr(n-r)$ pairs of $(\tilde{\Gamma},e)$. This means that the double sum is equal to
 
 $$[X(C_n)]S_1S_2\dots S_rL_rL_{n-1}^{2n-2-r}=\sum_{\substack{\tilde{\Gamma}\in\tilde{\Sigma}_n\\|E(\tilde{\Gamma})|=r}}\sum_{e\in \mathcal E_{\tilde{\Gamma}}}2 [Y_{\tilde{\Gamma}/e}]L_{n-2}^{2n-2-r}=$$
 $$= \sum_{\tilde{\Upsilon}\in\tilde{\Sigma}_{n-1}}\frac {2nr(n-r)}{n-1}
 [Y_{\tilde{\Upsilon}}]L_{n-2}^{2n-2-r}=\frac {2nr(n-r)}{n-1}[X(C_{n-1})]S_1\dots S_{r-1}L_{n-2}^{2n-2-r}.$$
 
 \end{proof}

Let us consider the following notations:
$$F(n, r):=\begin{cases} [X(C_n)]S_1\dots S_rL_{n-1}^{2n-1-r} & \text{if } 0\le r\le n-3\\
0 & \text{otherwise}
\end{cases}$$

$$H(n, r):=\begin{cases}
[X(C_n)]S_1\dots S_rS_{n-2}L_{n-1}^{2n-2-r}& \text{if } 0\le r\le n-3\\
0 & \text{otherwise}.
\end{cases}$$

Also let $q(n)=\frac{n+2}{4}{2n \choose n}-3\cdot 2^{2n-3}$.

\begin{Lemma} \label{lemma:FH-recurrence}
Let $n\ge 4$ and $0\le r\le n-3$. Then
$$F(n,r)=\frac{2nr}{n-1}\cdot F(n-1,r-1)+\frac 1{n-r}F(n,r+1)+\frac{(n-r-2)}{n-r}H(n,r)$$
\end{Lemma}
\begin{proof}
We expand $L_{n-1}$ using the equation (\ref{relations}) from Section~\ref{quadrics}:
\begin{align*}
F(n,r)&=[X(C_n)]S_1\dots S_rL_{n-1}^{2n-1-r}\\
&=\frac{1}{n-r}\left([X(C_n)]S_1\dots S_rL_rL_{n-1}^{2n-2-r}+\sum_{j=1}^{n-r-1}j[X(C_n)]S_1\dots S_rS_{r+j}L_{n-1}^{2n-1-r} \right)\\
&=\frac{1}{n-r}\Bigg([X(C_{n-1})]S_1\dots S_{r-1}L_{n-1}^{2n-2-r}+[X(C_n)]S_1\dots S_rS_{r+1}L_{n-1}^{2n-2-r}+\\
&\ +(n-r-2)[X(C_n)]S_1\dots S_rS_{n-2}L_{n-1}^{2n-2-r}\Bigg)\\
&=\frac{1}{n-r}\Bigg(\frac{2nr(n-r)}{n-1}\cdot F(n-1,r-1)+ F(n,r+1)+(n-r-2)H(n,r)\Bigg)\\
&=\frac{2nr}{n-1}\cdot F(n-1,r-1)+\frac 1{n-r}F(n,r+1)+\frac{(n-r-2)}{n-r}H(n,r).
\end{align*}
We used Lemma \ref{lemma:killL_i} for the first term and all terms of the sum are 0 by Lemmas \ref{lemma:0} and \ref{lemma:Pataki}. Note that for the border cases $r=0$ and $r=n-3$ the formula still holds even though some of the terms appearing in expansion of $L_{n-1}$ do not exist. If that is the case, the corresponding term $F(n-1,r-1)$ or $F(n,r+1)$ is defined to be 0. 
\end{proof}

In \cite{mateusz2}, the authors used the Lascoux coeffiecients $\psi_I$ to compute the intersections of divisors in $\CQ(V)$. For the definition we refer the reader to \cite{mateusz2} and \cite{LLT}. 
We will also use them later in this section. We will need the following result.

\begin{Lemma}\label{lemma:psi}
Let $0\le a<b$. Then $$\psi_{a,b}=\frac 12 \left({a+b+1 \choose a+1}+{a+b+1 \choose a+2}+\dots+{a+b+1 \choose b}\right).$$
\end{Lemma}
\begin{proof}
From \cite[Proposition A.15(5)]{LLT} and using the fact ${a+b\choose a}={a+b\choose b}$, we have 
\begin{align*}
    \psi_{a,b}&=\frac 12{a+b \choose a}+{a+b \choose a+1}+\dots+{a+b \choose b-1}+\frac 12{a+b \choose b}\\
    &=\frac12 \left({a+b \choose a}+{a+b \choose a+1}\right)+\frac12 \left({a+b \choose a+1}+{a+b \choose a+2}\right)+\dots+\frac12 \left({a+b \choose b-1}+{a+b \choose b}\right)\\
    &=\frac 12 \left({a+b+1 \choose a+1}+{a+b+1 \choose a+2}+\dots+{a+b+1 \choose b}\right).
\end{align*}
\end{proof}

Let $b_1\ge b_2$ and $a_1+\dots+a_k=b_1+b_2$. Then we define the coefficients $M(a_1,\dots,a_k;b_1,b_2)$ to be the numbers which satisfy the following identity:

$$s_{a_1}(x_1,x_2)s_{a_2}(x_1,x_2)\dots s_{a_k}(x_1,x_2)=\sum_{\substack{b_1<b_2\\
b_1+b_2=a_1+\dots+a_k}}M(a_1,\dots,a_k;b_1,b_2)s_{b_1,b_2}(x_1,x_2).$$

The most of the following proofs will be bijective proofs and we will be counting the number of paths in the grid with some properties. Our grid will be $\NN\times \NN$ and we denote the coordinates by $x$ and $y$.

\begin{Lemma}\label{lemma:pathidentity}
Let $r\le n-3$. Then the following identity holds:
$$\sum_{j=n-1-r}^{n-1} {2n-2-r \choose j}{j\choose n-1-r}{2n-2-r-j\choose n-1-r}=\frac{(2n-2-r)!2^r}{(n-r-1)!(n-r-1)!r!}.$$
\end{Lemma}
\begin{proof}
We will give a bijective combinatorial proof. Consider the set of two-colored paths in the grid with the following properties:
\begin{itemize}
    \item The path starts in $(0,0)$ and ends in one of the points $(n-1-r,n-1),(n-r,n-2),(n-r+1,n-3)\dots,(n-1,n-r-1)$.
    \item The path goes only up and right.
    \item Exactly $r$ steps of the path are colored red, the rest of the path is blue.
    \item There are exactly $n-r-1$ blue steps which go right and $n-r-1$ blue steps which go up.
\end{itemize}
We will show that both LHS and RHS represent the number of such paths.

LHS: We can fix the endpoint and denote it by $(j,2n-r-2-j)$, where $n-1-r\le j\le n-1$. Then we choose non-colored path from $(0,0)$ to $(j,2n-r-2-j)$ which goes only up and right. There are ${2n-2-r \choose j}$ options how to do it. Then we need to choose $n-r-1$ steps which will be colored blue from $j$ steps going right and also we need to choose $n-r-1$ steps which will be colored blue from $2n-r-2-j$ steps going up. Thus, for the coloring we have ${j\choose n-1-r}{2n-2-r-j\choose n-1-r}$ options. Together, the number of options is $$\sum_{j=n-1-r}^{n-1} {2n-2-r \choose j}{j\choose n-1-r}{2n-2-r-j\choose n-1-r},$$ which is precisely LHS.

RHS: 
We first choose the sequence of steps of length $2n-2-r$ with $n-r-1$ blue steps going up, $n-r-1$ blue steps going right and $r$ red steps. Note that we have not chosen the direction of the red steps yet. For this we clearly have $\frac{(2n-2-r)!}{(n-r-1)!(n-r-1)!r!}$ options. Next we decide for each red step whether it goes up or right, we multiply the number by $2^r$. 
We get exactly the paths with the four properties we mentioned at the beginning. This way we see that the number of such paths is equal to $\frac{(2n-2-r)!2^r}{(n-r-1)!(n-r-1)!r!}$ which concludes the lemma.
\end{proof}

\begin{Lemma}\label{lemma:M}
Let $a_1+\dots+a_k=b_1+b_2$ and $b_1\ge b_2$. Then $M(a_1,\dots,a_k;b_1,b_2)$ represents the number of all paths in the grid with marked points with the following properties:
\begin{itemize}
    \item The path starts in $(0,0)$ and ends in $(b_1,b_2)$.
    \item The paths goes only right and up, and whole path is under diagonal $x=y$.
    \item Exactly $k+1$ points on the path are marked, including the points $(0,0)$ and $(b_1,b_2)$.
    \item Between two marked points the path go first up, and then right. It is possible that the path goes only up or only right between two marked points.
    \item Between $i$-th and $(i+1)$-th marked points in the path there are exactly $a_i$ steps.
\end{itemize}
\end{Lemma}

\begin{proof}
For the multiplication of Schur polynomials we apply the Pieri's rule. Thus, in the expansion of $s_{a_1}\dots s_{a_k}$ the coefficient of $s_{b_1,b_2}$ is the number of ways of filling the Young diagram $(b_1,b_2)$ with numbers $1,\dots,k$ such that in the rows the numbers are weakly increasing, in the columns the numbers are increasing and the number $i$ is used exactly $a_i$ times. 

We can interpret this as the path in the grid as follows: We look at the number of ones in the second row and make that many steps up. Then we look at the number of ones in the first row and make that many steps right. Then we mark the point we end on and repeat the procedure with $2,3\dots,k$.

We claim that the path constructed in this way lies under the diagonal $x=y$. Indeed, this follows from the fact that the columns are increasing, or in the other words, when we are multiplying Young diagrams, we can not have at any point more boxes in the second row than in the first. Moreover, the path satisfies all other required properties from the statement.

On the other hand, from every such path we can easily recover the filling of the Young diagram, which proves the lemma.
\end{proof}

\begin{Lemma}\label{lemma:multiplyingschurrows}
Let $r\le n-3$ and $n-r-1\le j\le n-1$. Then the following identity holds:
$$\sum_{\substack{\Gamma\in\Sigma'_n\\|E(\Gamma)|=r}}\sum_{\substack{b_1\ge b_2\ge 0\\b_1+b_2=r\\b_2\le n-1-j,j-(n-1-r)}} M(\gamma_1,\dots,\gamma_k;b_1,b_2)={j\choose n-1-r}{2n-2-r-j\choose n-1-r}.$$
\end{Lemma}
\begin{proof}
Clearly, if we replace $j$ by $2n-2-r-j$, then the values of both sides do not change. Thus, we may assume $j\le (2n-2-r)/2$.

We consider the set $\mathcal P$ of all paths in the grid, with marked points with assigned numbers with the following properties:
\begin{itemize}
    \item The path starts in $(0,0)$ and ends in $(b_1,b_2)$, such that $b_1+b_2=r$ and $b_2\le n-1-j$.
    \item The paths goes only right and up, and whole path is under the diagonal $x=y$.
    \item Some of the points on the path are marked, including the points $(0,0)$ and $(b_1,b_2)$.
    \item Between two marked points the path go first up, and then right. It is possible that the path goes only up or only right between two marked points.
    \item To every marked point is assigned a non-negative integer such that the sum of all assigned integers is $n-1-r$. Moreover, this integer is positive on all marked points in the middle of the path, i.e. it can be 0 only in $(0,0)$ and $(b_1,b_2)$.
\end{itemize}

We show that LHS represents the number of such paths. 

Firstly, we pick the end point $(b_1,b_2)$, then we choose the number of marked points, which we denote by $k+1$. Clearly, $k\le n-r$, otherwise it would not be possible to assign the integers. Then we choose the numbers $a_1,\dots,a_k$ which represents the length of the path between the marked points. Then we choose the non-negative integers $c_1,\dots,c_{k+1}$ such that $c_1+\dots+c_{k+1}=n-1-r$ and $c_2,\dots,c_k>0$, which will be the assigned integers to the marked points. Then by Lemma~\ref{lemma:M}, the number of such paths without assigned integers is $M(a_1,\dots,a_k;b_1,b_2)$. Thus together the number of such paths is $$\sum_{k=1}^{n-r}\sum_{a_1+\dots+a_k=r}\ \sum_{\substack{c_1+\dots+c_{k+1}=n-1-r\\
c_2,\dots,c_k>0}}\sum_{\substack{b_1\ge b_2\ge 0\\b_1+b_2=r\\b_2\le n-1-j,j-(n-1-r)}} M(\gamma_1,\dots,\gamma_k;b_1,b_2)$$

We claim that every choice of $a_1,\dots,a_k,c_1,\dots,c_{k+1}$ corresponds to a graph $\Gamma\in\Sigma'_n$ with $\gamma_i=a_i$. Indeed, given the numbers $a_1,\dots,a_k,c_1,\dots,c_{k+1}$ we can construct the graph $\Gamma\in\Sigma'_n$ as follows: We look at $\Gamma$ as the subgraph of the path of the length $n-1$ and we start drawing this graph from one end. The first $c_1$ edges will not be the edges of $\Gamma$. Then the next $a_1$ edges will be edges of $\Gamma$, next $c_2$ edges will not be edges of $\Gamma$ and so on. It has $r$ edges and $\gamma_i=a_i$. Moreover, from the graph $\Gamma\in\Sigma'_n$ we can get back the numbers $a_1,\dots,a_k,c_1,\dots,c_{k+1}$.

Thus the number of elements in $\mathcal P$ is equal to LHS.

Next we consider another set $\mathcal P'$ of paths with marked points and assigned numbers and the following properties:

\begin{itemize}
    \item The path starts in $(0,0)$ and ends in $(r-n+1+j,n-1-j)$.
    \item The path goes only right and up, but it may goes above the diagonal.
    \item Some of the points on the path are marked, including the points $(0,0)$ and $(r-n+1+j,n-1-j)$.
    \item Between two marked points the path go first up, and then right. It is possible that the path goes only up or only right between two marked points.
    \item To every marked point is assigned a non-negative integer such that the sum of all assigned integers is $n-1-r$. Moreover, this integer is positive on all marked points in the middle of the path, i.e. it can be 0 only in $(0,0)$ and $(r-n+1+j,n-1-j)$.
\end{itemize}

We claim the number of such paths is the same as the number of pairs of binary sequences $(U,R)$, such that the sequence $U$ has $n-1-j$ ones and $n-1-r$ zeros and $R$ has $r-n+1+j$ ones and $n-1-r$ zeros. The number of such pairs of sequences is ${j\choose n-1-r}{2n-2-r-j\choose n-1-r}=RHS$. 

Consider a path from $\mathcal P'$. We construct the sequence $U$ as follows: We travel on our path from $(0,0)$ to $(r-n+1+j,n-1-j)$ and every time we find a marked point we write that many zeros to our sequence as is the assigned integer to that marked point. Everytime we find a step up, we write one in our sequence. Similarly, we construct a sequence $R$, the only difference is that we write one for every step right. Clearly, the sequences constructed in this way have the desired number of zeros and ones.

On the other hand, we consider a pair of binary sequences $(U,R)$. We denote by $u_1$ the number of ones in front of the first zero in $U$. Then we denote $u_2$ the number of ones between the first an the second zero in $U$ and so on, with $u_{n-r}$ being the number of ones behind the last zero in $U$. Analogously, we define $\rho_1,\dots \rho_{n-r}$. We construct a path as follows: We mark point $(0,0)$ and assign 0 to it. We make $u_1$ steps up and $\rho_1$ steps right, then we mark the last point and assign 1 to it. Then 
we make $u_2$ steps up and $\rho_2$ steps right and we mark the last point and assign 1 to it, and so on. There is one exception. Every time we have $u_i=\rho_i=0$ we obviously do not make any steps and there is no new point to be marked. When this happens we instead increase the assigned number to the last point by 1. Clearly, this way we obtain a path from $\mathcal P'$.

Moreover, it is not difficult to check that these constructions are inverses of each other and thus they are bijections between the paths and binary sequences.
Thus, the number of elements of $\mathcal P'$ is represented by the RHS of the equation from the statement.

To conclude the lemma, we need to show that the number of paths in $\mathcal P$ is the same as the number of paths in $\mathcal P'$. Again, we will give a bijective proof.

Consider a path from $\mathcal P'$. If the path is under the diagonal we do not change it and we have a path from $\mathcal P$. Otherwise, we look at the first point on the path that is most far away from the diagonal, i.e. on the first point on the path with maximal difference $y-x=a$. The step which got us here must be step up. Now we change the direction of this step, i.e. we go right instead of going up, and we do not change the direction of the other steps. We also do not change the marked points and the assigned numbers. Note that the next step is going right, otherwise we would have a point which is even further away from the diagonal. Thus the condition, that between two marked points we go first up and then right is still satisfied. Moreover our path is now under the line $y=x+a-1$ and ends in $(r-n+2+j,n-2-j)$.

We continue similarly: We find the first point on the new path which lies on the line $y=x+a-1$. Again, the step which got us here is the step up and we change the direction of this step. We obtain a path which is under the line $y=x+a-2$ and we continue until we get a path under the diagonal $x=y$. Clearly, this path is from $\mathcal P$. Thus, we assign to each path from $\mathcal P'$ a path from $\mathcal P$.

Now we consider a path from $\mathcal P$, which ends in $(b_1,b_2)$. If $b_2=n-1-j$ we do not change it and we have a path from $\mathcal P'$. Otherwise, we look at the last point on the path which is on the diagonal $y=x$. The step from this point must go right and we change its direction. This means this step goes up. We do not change the direction of the other steps, marked points and assigned numbers. Note that the previous step must be the step up, since our path is under the diagonal $x=y$. Thus, we do not break the condition about the path between two marked points. This way we obtain a path which ends in $(b_1-1,b_2+1)$ and is under the line $y=x+1$. If $b_2+1=n-1-j$ we finish and we have the path from $\mathcal P'$. Otherwise we continue. We find the last point on the new path which is on the line $y=x+1$. The step from this point goes right and we change its direction. We obtain a path which ends in $(b_1-2,b_2+2)$ and is under the line $y=x+2$. Then we look at the last point on this line and so on. We keep going until we obtain a path which ends in $(r-n+1+j,n-1-j)$. Clearly, this path will be from $\mathcal P'$.

Thus, we can associate to a path from $\mathcal P$ a path from $\mathcal P'$ and vice versa. It is straightforward to check that the two constructions above are inverses to each other which means that we constructed a bijection between $\mathcal P$ and $\mathcal P'$ which proves the lemma.
\end{proof}

\begin{Example}

We illustrate the proof of Lemma~\ref{lemma:multiplyingschurrows} in the following example by taking $r=9$, $n=15$, $j=10$, $b_1=7$, $b_2=2$.

Consider the following subgraph $\Gamma$ of $C_{15}$ in $\Sigma'_{15}$ with $\gamma_1=\gamma_2=3,\gamma_3=1,\gamma_4=2$.

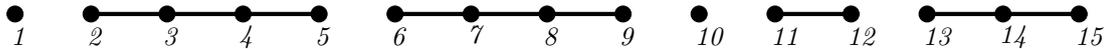
\begin{figure}[H]
\begin{tikzpicture}
\foreach \x in {1,2,...,15}{
\filldraw [black] (\x,0) circle (3pt);}
\draw [black, ultra thick] (2,0)--(5,0); 
\draw [black, ultra thick] (6,0)--(9,0);
\draw [black, ultra thick] (11,0)--(12,0);
\draw [black, ultra thick] (13,0)--(15,0); 
\foreach \x in {1,2,...,15}{
\node [right, black, ultra thick] at (\x-0.2,-0.3) {\x};}
\end{tikzpicture}
\end{figure}

 Consider the following term from the multiplication of Young diagrams corresponding to the graph $\Gamma$:

\ytableausetup{centertableaux}
\[\begin{ytableau}
1 & 1 & 1 \\
\end{ytableau}
\cdot
\begin{ytableau}
2 & 2 & 2 \\
\end{ytableau}
\cdot
\begin{ytableau}
3 \\
\end{ytableau}
\cdot
\begin{ytableau}
4 & 4 \\
\end{ytableau}
\leadsto
\begin{ytableau}
1 & 1 & 1 & 2& 2 \\
2&3&4&4
\end{ytableau}\]

To this graph and Young diagrams correspond a path from $\mathcal P$, which is the blue path on the Figure~\ref{fig}. To this path correspond a path from $\mathcal P'$ which is the red path on the Figure~\ref{fig}.
\begin{figure}[H]
\begin{tikzpicture}
\tikzstyle{every node}=[font=\Large]
\draw [red, ultra thick](0,0) --(0,1) -- (1,1) -- (2,1) -- (2,2)-- (2,3)-- (2,4)-- (3,4)-- (4,4)--(5,4);
\draw [blue, ultra thick](0,0) --(3,0) -- (3,2) -- (7,2);
\filldraw [black] (0,0) circle (3pt);
\filldraw [black] (3,0) circle (3pt);
\filldraw [black] (4,2) circle (3pt);
\filldraw [black] (5,2) circle (3pt);
\filldraw [black] (7,2) circle (3pt);
\filldraw [black] (2,1) circle (3pt);
\filldraw [black] (2,4) circle (3pt);
\filldraw [black] (3,4) circle (3pt);
\filldraw [black] (5,4) circle (3pt);
\node [right, black, ultra thick] at (0,-0.3) {1};
\node [right, black, ultra thick] at (2,0.7) {1};
\node [right, black, ultra thick] at (2,3.7) {2};
\node [right, black, ultra thick] at (3,3.7) {1};
\node [right, black, ultra thick] at (5,3.7) {0};
\node [right, black, ultra thick] at (3,-0.3) {1};
\node [right, black, ultra thick] at (4,1.7) {2};
\node [right, black, ultra thick] at (5,1.7) {1};
\node [right, black, ultra thick] at (7,1.7) {0};

\draw [black, dashed, ultra thick](0,0) --(1,1) -- (2,2) -- (3,3) -- (4,4)-- (5,5);
\draw [black, thin] (3,0)--(7,0);
\draw [black, thin] (2,0)--(2,1)--(7,1);
\draw [black, thin] (0,2)--(3,2)--(3,5);
\draw [black, thin] (0,3)--(7,3);
\draw [black, thin] (0,1)--(0,4)--(2,4)--(2,5);
\draw [black, thin] (5,4)--(7,4);
\foreach \x in {1,4,5,6,7}{
\draw [black, thin] (\x,0)--(\x,5);}
\draw [black, thin] (0,4)--(0,5)--(7,5);
\end{tikzpicture}
\centering
\caption{A path in $\mathcal P$ and a path in $\mathcal P'$}\label{fig}
\end{figure}
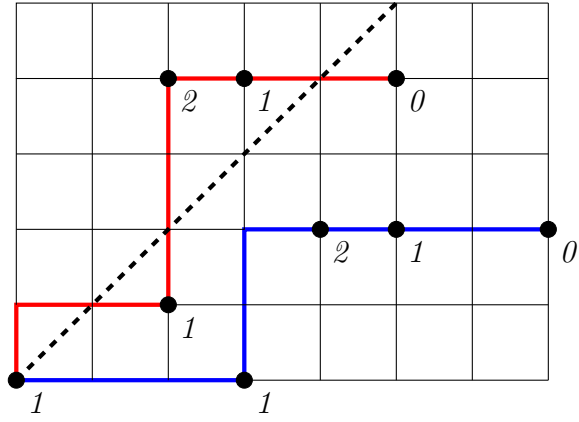

To the red path from $\mathcal P' $ corresponds the pair of binary sequences $U=010111000$, $R=0110001011$.
\end{Example}

\begin{Theorem}\label{lemma:Hformula}
Let $0\le r\le n-3$. Then the following formula holds:
$$H(n,r)=\frac{(2n-r-2)!\cdot n\cdot 2^{r-1}}{(n-r)!(n-r-1)!}$$
\end{Theorem}
\begin{proof}
By Lemma~\ref{lemma:XE1..ErEn-2} we know that $X(C_n)\cap E_1\cap\dots\cap E_r\cap E_{n-2}$ has $n(n-1)\dots(n-r+1)$ components which corresponds to graphs $\tilde{\Gamma}\in \tilde{\Sigma}_n$. Thus, we need to just sum up over these components. Let us fix one of them. By Lemma~\ref{lemma:XE1..ErEn-2}, there is a map from this component to $\mathcal Q_{\tilde{\Gamma}}$ which is an isomorphism on an open subset of $\mathcal Q_{\tilde{\Gamma}}$.

We will compute in the Chow ring of Grassmannian $G(n-2,n)$. We know that the class of $G_{\tilde{\Gamma}}$ is given by $s_{1,\dots,1}\dots s_{1,\dots,1}$ where the number of ones are $\gamma_1,\dots,\gamma_k$. 

To simplify the notation, we will transpose every partition in this proof, which will not affect the final result. Thus we will work with partitions inside a $2\times(n-2)$ rectangle, and the class of $G_{\tilde{\Gamma}}$ is given by $s_{\gamma_1}\dots s_{\gamma_k}$. $L_{n-1}$ is just the hyperplane section in the $S^2(\mathcal Q)$. Thus, by intersecting with the power of $L_{n-1}$ we obtain the Segre classes of $S^2(\mathcal Q)$, which are described by Lascoux coefficients $\psi_{b_1,b_2}$, see~\cite{mateusz2}.

Since we are intersecting with general sections of $S^2(\mathcal Q)$ and we are obtaining just points, it is clear that we can assume that these points lie only in the open subset of $\mathcal Q_{\tilde{\Gamma}}$ which is isomorphic to the corresponding component of $X(C_n)\cap E_1\cap\dots\cap E_r\cap E_{n-2}$.

Thus, the intersection in the component corresponding to $\tilde{\Gamma}$ is equal to

$$s_{\gamma_1}\dots s_{\gamma_k}\Seg_{2n-4-r}(S^2(\mathcal Q))=s_{\gamma_1}\dots s_{\gamma_k}\sum_{b_1+b_2=2n-4-r,b_1\ge b_2}s_{b_1,b_2}\psi_{b_2,b_1+1}.$$

Moreover, the complementary partitions get multiplied to $s_{n-2,n-2}$ which is the class of a point, and otherwise the product is 0. Thus, by using Lemma \ref{lemma:psi}, we obtain:

\begin{align*}
    H(n,r)&=\sum_{\substack{\tilde{\Gamma}\in\tilde{\Sigma}_n\\
|E(\tilde{\Gamma})|=r}}\sum_{\substack{b_1\ge b_2\\
b_1+b_2=r}} M(\gamma_1,\dots,\gamma_k;b_1,b_2)\psi_{n-2-b_1,n-1-b_2}\\
&=\sum_{\substack{\Gamma\in\Sigma_n\\
|E(\Gamma)|=r}}r!\sum_{\substack{b_1\ge b_2\\
b_1+b_2=r}} M(\gamma_1,\dots,\gamma_k;b_1,b_2)\cdot \frac{1}{2}\sum_{j=n-1-b_1}^{n-1-b_2}{2n-2-r \choose j}
\end{align*}

When we act with $\sigma\in\ZZ_n\subset\Aut(C_n)=D_n$ on our graph $\Gamma$ the numbers $\gamma_1,\dots,\gamma_k$ do not change. Thus, without loss of generality, we may assume that $\Gamma$ does not contain the edge $(1,n)$ by rotating the graph $\Gamma$ such that it does not contain it. For any graph with $r$ edges there are $n-r$ possibilities how to do it. On the other hand, for any graph $\Gamma\in\Sigma_n$ which does not contain $(1,n)$ there are $n$ graphs which can be rotated to this one. Thus we can change the sum, to sum only through graphs in $\Sigma'_n$, but we have to multiply it by $n/(n-r)$. 
Moreover, we also change the order of summation and sum through $j$ first to obtain by Lemmas~\ref{lemma:multiplyingschurrows} and \ref{lemma:pathidentity} the following:

\begin{align*}
    H(n,r)&=\frac{r!n}{2(n-r)}\sum_{\substack{\Gamma\in\Sigma'_n\\|E(\Gamma)|=r}}\sum_{\substack{b_1\ge b_2\ge 0\\b_1+b_2=r}} M(\gamma_1,\dots,\gamma_k;b_1,b_2)\cdot \sum_{j=n-1-b_1}^{n-1-b_2}{2n-2-r \choose j}\\
    &=\frac{r!n}{2(n-r)}\sum_{j=n-1-r}^{n-1} {2n-2-r \choose j}\sum_{\substack{\Gamma\in\Sigma'_n\\|E(\Gamma)|=r}}\sum_{\substack{b_1\ge b_2\ge 0\\b_1+b_2=r\\b_2\le n-1-j,j-(n-1-r)}} M(\gamma_1,\dots,\gamma_k;b_1,b_2)\\
    &=\frac{r!n}{2(n-r)}\sum_{j=n-1-r}^{n-1} {2n-2-r \choose j}{j\choose n-1-r}{2n-2-r-j\choose n-1-r}\\
    &=\frac{r!n}{2(n-r)}\frac{(2n-2-r)!2^r}{(n-r-1)!(n-r-1)!r!}=\frac{(2n-2-r)!n2^{r-1}}{(n-r)!(n-r-1)!}.
\end{align*}
which proves the theorem. 
\end{proof}

\begin{Theorem}\label{thm:Fformula}
Let $0\le r\le n-2$. Then the following formula holds:
$$F(n,r)=q(n-r)\cdot \frac {2^{r+1}\cdot n\cdot(2n-r-1)!}{(2n-2r)!}$$
\end{Theorem}
\begin{proof}
We prove it by induction on $n$ and $r$. For $r=n-2$ we have $F(n,n-2)=0$ by definition and also $q(2)=\frac{2+2}{4}{4\choose 2}-3\cdot 2=0$, thus the theorem holds. 

Now we assume that the statement holds for $(n',r')$ with $n'<n$ or $n'=n, r<r'\le n-3$. We use Lemma \ref{lemma:FH-recurrence}:
\begin{align*}
F(n,r)&=\frac{2nr}{n-1}\cdot F(n-1,r-1)+\frac 1{n-r}F(n,r+1)+\frac{n-r-2}{n-r}H(n,r)\\
&=\frac{2nr}{n-1}q(n-r)\frac {2^{r}\cdot (n-1)\cdot(2n-r-2)!}{(2n-2r)!}+q(n-r-1)\frac {2^{r+2}\cdot n\cdot(2n-r-2)!}{(n-r)(2n-2r-2)!}+\\
&\ +\frac{(n-r-2)(2n-r-2)!\cdot n\cdot 2^{r-1}}{(n-r)!(n-r-1)!(n-r)}\\
&=\frac {2^{r+1}\cdot n\cdot(2n-r-1)!}{(2n-2r)!}\Bigg(q(n-r) \frac{r}{(2n-r-1)}+q(n-r-1)\frac{4(2n-2r-1)}{(2n-r-1)}+\\
&\ +\frac{(n-r-2)(2n-2r)!}{4(2n-r-1)(n-r)!(n-r)!}\Bigg)\\
&=\frac {2^{r+1}\cdot n\cdot(2n-r-1)!}{(2n-2r)!}\Bigg(\bigg(\frac{n-r+2}{4}{2n-2r \choose n-r}-3\cdot 2^{2n-2r-3}\bigg)\frac{r}{(2n-r-1)}+\\
&\ + \bigg(\frac{n-r+1}{4}{2n-2r-2 \choose n-r-1}-3\cdot 2^{2n-2r-5}\bigg)\frac{4(2n-2r-1)}{(2n-r-1)}+{2n-2r \choose n-i}\frac{(n-r-2)}{4(2n-r-1)}\Bigg)\\
&=\frac {2^{r+1}\cdot n\cdot(2n-r-1)!}{(2n-2r)!}\Bigg(-3\cdot 2^{2n-2r-3}\cdot\bigg(\frac{r}{(2n-r-1)}+\frac{(2n-2r-1)}{(2n-r-1)}\bigg)+\\
&\ +\frac 14 {2n-2r\choose n-r} \bigg(\frac{r(n-r+2)}{2n-r-1}+\frac{2(n-r+1)(n-r)}{2n-r-1}+\frac{n-r-2}{2n-r-1} \bigg)\Bigg)\\
&=\frac {2^{r+1}\cdot n\cdot(2n-r-1)!}{(2n-2r)!}\Bigg(-3\cdot 2^{2n-2r-3}\cdot\frac{r+2n-2r-1}{2n-r-1}+\\
&\ +\frac 14 {2n-2r\choose n-r} \bigg(\frac{nr-r^2+2r+2n^2-4nr+2r^2+2n-2r+n-r-2}{2n-r-1}\bigg)\Bigg)\\
&=\frac {2^{r+1}\cdot n\cdot(2n-r-1)!}{(2n-2r)!}\Bigg(-3\cdot 2^{2n-2r-3}+\frac 14 {2n-2r\choose n-r}\frac{(n-r+2)(2n-r-1)}{(2n-r-1)}\Bigg)\\
&=q(n-r)\cdot \frac {2^{r+1}\cdot n\cdot(2n-r-1)!}{(2n-2r)!}.
\end{align*}
Note that in the case $r=0$, the term $\frac{2nr}{n-1}\cdot F(n-1,r-1)=0$ and it remains 0 when we plug in anything for $F(n-1,r-1)$. Thus, we may plug in the expression from the inductive hypothesis even though it is not true (formula for $r=-1$ does not return zero).
\end{proof}

From the Theorem \ref{thm:Fformula} we can easily conclude the Theorem~\ref{variety}:
\begin{proof}[Proof of Theorem \ref{variety}]
By Lemma \ref{CQ_formula}, we have that $\deg(L(C_n)^{-1})=[X(C_n)]L_{n-1}^{2n-1}$. By Theorem \ref{thm:Fformula}, we obtain 
\[\deg\Big(L(C_n)^{-1}\Big)=F(n,0)=q(n)\frac {2n\cdot(2n-1)!}{(2n)!}=q(n)=\frac{n+2}{4}{2n \choose n}-3\cdot 2^{2n-3}.\]

\end{proof}

\end{document}